\pdfoutput=1
\documentclass[11pt,a4paper,reqno]{amsart}

\usepackage[british]{babel}

\usepackage[DIV=9,oneside,BCOR=0mm]{typearea}
\usepackage{microtype}
\usepackage[T1]{fontenc}
\usepackage{setspace}
\usepackage{amssymb}
\usepackage{amsmath}
\usepackage{amsthm}
\usepackage{enumitem}
\usepackage{mathrsfs}
\usepackage[colorlinks,citecolor=blue,urlcolor=blue,linkcolor=blue,linktocpage]{hyperref}
\usepackage{mathrsfs}
\usepackage{mathtools}
\usepackage{xcolor}
\usepackage{xfrac}

\newtheorem{innercustomgeneric}{\customgenericname}
\providecommand{\customgenericname}{}
\newcommand{\newcustomtheorem}[2]{%
	\newenvironment{#1}[1]
	{%
		\renewcommand\customgenericname{#2}%
		\renewcommand\theinnercustomgeneric{##1}%
		\innercustomgeneric
	}
	{\endinnercustomgeneric}
}

\newcustomtheorem{customthm}{Theorem}
\newcustomtheorem{customcor}{Corollary}

\newtheorem{thm}{Theorem}[section]

\newtheorem{cor}[thm]{Corollary}
\newtheorem{lem}[thm]{Lemma}
\newtheorem{prop}[thm]{Proposition}
\newtheorem{problem}[thm]{Problem}

\theoremstyle{definition}

\newtheorem{definition}[thm]{Definition}

\newtheorem{remark}[thm]{Remark}

\renewcommand{\epsilon}{\varepsilon}
\renewcommand{\phi}{\varphi}
\newcommand{\defeq}{\mathrel{\mathop:}=}
\newcommand{\eqdef}{\mathrel{\mathopen={\mathclose:}}}
\renewcommand{\Re}{\operatorname{Re}}

\DeclareMathOperator{\CM}{C}
\DeclareMathOperator{\SF}{S}
\DeclareMathOperator{\vN}{L}
\DeclareMathOperator{\id}{id}
\DeclareMathOperator{\pr}{pr}
\DeclareMathOperator{\spt}{spt}
\DeclareMathOperator{\gr}{gr}
\DeclareMathOperator{\B}{B}
\DeclareMathOperator{\U}{U}
\DeclareMathOperator{\OG}{O}
\DeclareMathOperator{\im}{im}
\DeclareMathOperator{\Aut}{Aut}
\DeclareMathOperator{\Homeo}{Homeo}
\DeclareMathOperator{\Cstar}{C^{\ast}}
\DeclareMathOperator{\Cont}{C}
\DeclareMathOperator{\Hom}{Hom}
\DeclareMathOperator{\E}{\mathbb{E}}
\DeclareMathOperator{\C}{\mathbb{C}}
\DeclareMathOperator{\R}{\mathbb{R}}

\DeclareMathOperator{\Z}{\mathbb{Z}}
\DeclareMathOperator{\N}{\mathbb{N}}
\DeclareMathOperator{\T}{\mathbb{T}}

\DeclareMathOperator{\Neigh}{\mathcal{U}}
\DeclareMathOperator{\Fix}{Fix}
\DeclareMathOperator{\Tor}{Tor}

\makeatletter
\def\moverlay{\mathpalette\mov@rlay}
\def\mov@rlay#1#2{\leavevmode\vtop{%
		\baselineskip\z@skip \lineskiplimit-\maxdimen
		\ialign{\hfil$\m@th#1##$\hfil\cr#2\crcr}}}
\newcommand{\charfusion}[3][\mathord]{
	#1{\ifx#1\mathop\vphantom{#2}\fi
		\mathpalette\mov@rlay{#2\cr#3}
	}
	\ifx#1\mathop\expandafter\displaylimits\fi}
\makeatother

\newcommand{\bigveedot}{\charfusion[\mathop]{\bigvee}{\boldsymbol{\cdot}}}

\begin{document}
%%%%%%%%%%%%%%%%%%%%%%%%%%%%%%%%

\setlist{noitemsep}

\author{Yannik Höll}
\address{Y.H., Institute of Discrete Mathematics and Algebra, TU Bergakademie Freiberg, 09596 Freiberg, Germany}
\email{yannik.hoell@math.tu-freiberg.de}
\author{Friedrich Martin Schneider}
\address{F.M.S., Institute of Discrete Mathematics and Algebra, TU Bergakademie Freiberg, 09596 Freiberg, Germany}
\email{martin.schneider@math.tu-freiberg.de}
\thanks{This research is funded by the Deutsche Forschungsgemeinschaft (DFG, German Research Foundation) -- Projektnummer 561178190}

\title[Whirliness beyond amenability]{Whirliness beyond amenability}
\date{\today}

\begin{abstract}
	Given a non-atomic standard probability space $(\Omega,\mu)$, we exhibit examples of non-amenable closed topological subgroups of $\Aut(\Omega,\mu)$ whose natural near-action on $(\Omega,\mu)$ is whirly. This answers a 2010 question by Pestov in the negative. The argument proceeds via constructing Gaussian near-actions from topologically faithful unitary representations of non-amenable whirly Polish groups.
\end{abstract}



\maketitle

\allowdisplaybreaks

%%%%%%%%%%%%%%%%%%%%%%%%%%%%%%%%%%%%%%%%%%
%%%%%%%%%%%%%%%%%%%%%%%%%%%%%%%%%%%%%%%%%%

%\tableofcontents

%\newpage

%\vspace{-5mm}

\section{Introduction}

The automorphism group $\Aut(\Omega,\mu)$ of the probability space $(\Omega,\mu)$, i.e., the group all $\mu$-equivalence classes of $\mu$-preserving automorphisms of the measurable space~$\Omega$, naturally acts by isometries on the metric space $(\mathcal{B}_{\mu},d_{\mu})$, where $\mathcal{B}_{\mu}$ denotes the quotient of the $\sigma$-algebra of measurable subsets of $\Omega$ modulo its $\sigma$-ideal of $\mu$-null sets and where \begin{displaymath}
	d_{\mu} \colon \, \mathcal{B}_{\mu} \times \mathcal{B}_{\mu} \, \longrightarrow \, [0,1] , \quad (A,B) \, \longmapsto \, \mu(A \triangle B) .
\end{displaymath} In turn, equipped with the \emph{weak topology}, that is, the initial topology generated by the maps of the form \begin{displaymath}
	\Aut(\Omega,\mu) \, \longrightarrow \, \mathcal{B}_{\mu}, \quad T \, \longmapsto \, T(A) \qquad (A \in \mathcal{B}_{\mu}) ,
\end{displaymath} $\Aut(\Omega,\mu)$ constitutes a topological group. A \emph{near-action} $\alpha$ of a group $G$ on $(\Omega,\mu)$, often denoted by $G \curvearrowright^{\alpha} (\Omega,\mu)$, is a homomorphism $\alpha \colon G \to \Aut(\Omega,\mu)$. A near-action of a topological group on $(\Omega,\mu)$ is called \emph{weakly continuous} if the corresponding homomorphism to $\Aut(\Omega,\mu)$ is continuous with respect to the weak topology.

%\enlargethispage{8mm}

A weakly continuous near-action $G \curvearrowright^{\alpha} (\Omega,\mu)$ of a Polish group $G$ on a standard probability space $(\Omega,\mu)$ is called \emph{whirly}~\cite[Definition~3.2]{GlasnerTsirelsonWeiss} if, for any two measurable $A,B \subseteq \Omega$ with $\mu(A)\mu(B) > 0$, the set $\{ g \in G \mid \mu(A \cap gB) > 0 \}$ has non-empty intersection with every neighborhood of the neutral element in $G$. This property prevents a non-trivial near-action from admitting a \emph{spatial model}~\cite[Proposition~3.3(b)]{GlasnerTsirelsonWeiss}. Since every weakly continuous near-action of a second-countable locally compact group admits such a spatial model by work of Mackey, Varadarajan, and Ramsay~\cite[Theorem~0.3(a)]{GlasnerTsirelsonWeiss}, this means that whirliness is a dynamical peculiarity belonging to the theory of non-locally compact Polish groups.

Identifying a rich source of the phenomenon introduced above, Glasner, Tsirelson and Weiss have shown that if a Polish group $G$ is \emph{L\'evy}, i.e., there exists an ascending chain $(K_{n})_{n \in \N}$ of compact subgroups of $G$ such that \begin{itemize}
	\item[---\,] $\bigcup_{n \in \N} K_{n}$ is dense in $G$, and
	\item[---\,] $(K_{n},d\vert_{K_{n}^{2}},\mu_{n})_{n \in \N}$ is a \emph{L\'evy family}~\cite{GromovMilman}, where $d$ is a right-invariant compatible metric on $G$ and $\mu_{n}$ denotes the normalized Haar measure on $K_{n}$ $(n \in \N)$,
\end{itemize} then $G$ itself is \emph{whirly}, in the sense that every weakly continuous ergodic near-action of $G$ on a probability space is whirly (see~\cite[Theorem~1.1]{GlasnerTsirelsonWeiss} and~\cite[Theorem~3.10]{GlasnerWeiss}). Their result naturally led to the following question.

\begin{problem}[Glasner--Weiss~{\cite[\S6, p.~1537, Problem~2]{GlasnerWeiss}}]\label{problem:glasner.weiss} Let $(\Omega,\mu)$ be a non-atomic standard probability space and let $G$ be a closed topological subgroup of $\Aut(\Omega,\mu)$ such that $G \curvearrowright (\Omega,\mu)$ is whirly. Is $G$ necessarily a L\'evy group? \end{problem}

The question raised in Problem~\ref{problem:glasner.weiss} was answered in the negative by Pestov~\cite{pestov10}. But while being non-L\'evy, Pestov's examples still retained the property of \emph{amenability}.

A topological group is called \emph{amenable} if every continuous action of it on a non-void compact Hausdorff space admits an invariant regular Borel probability measure. For instance, any topological group containing a chain of compact subgroups whose union is everywhere dense must be amenable. What is more, if a topological group is amenable and whirly (e.g., due to being a L\'evy group), then it is even \emph{extremely amenable}, i.e., it actually admits a fixed point whenever it acts continuously on a non-empty compact Hausdorff space~\cite[Corollary~3.4]{pestov10}. On the other hand, extreme amenability does not necessitate whirliness~\cite[Remark~1.3]{GlasnerTsirelsonWeiss}.

In view of the above, Pestov asked the following question.

\begin{problem}[Pestov~{\cite[\S7, Problem~(2)]{pestov10}}]\label{problem:pestov} Let $(\Omega,\mu)$ be a non-atomic standard probability space and let $G$ be a closed topological subgroup of $\Aut(\Omega,\mu)$ such that $G \curvearrowright (\Omega,\mu)$ is whirly. Is $G$ necessarily amenable? \end{problem}

The present manuscript solves Problem~\ref{problem:pestov}. Inspired by Pestov's approach~\cite{pestov10} to Problem~\ref{problem:glasner.weiss}, our solution of Problem~\ref{problem:pestov} makes use of the following theorem, essentially an application of Gaussian near-actions.

\begin{thm}\label{theorem:construction} If $G$ is a unitarily representable, second-countable topological group, then there exist a standard probability space $(\Omega,\mu)$ and a topological group embedding $G \hookrightarrow \Aut(\Omega,\mu)$ such that the induced near-action $G \curvearrowright (\Omega,\mu)$ is ergodic. \end{thm}

Problem~\ref{problem:pestov} can be solved by applying Theorem~\ref{theorem:construction} to any unitarily representable non-amenable whirly Polish group. Through Theorem~\ref{theorem:examples} below, we produce several natural examples of such objects (see Section~\ref{section:examples} for details). Moreover, our strategy for establishing whirliness for these examples turns out to yield a proof of another peculiar dynamical property of large topological groups---\emph{inertness} (see Section~\ref{section:inertness}).

\begin{thm}\label{theorem:examples} \begin{enumerate}
	\item\label{theorem:examples.1} Let $\mu$ be a diffuse submeasure and let $G$ be a topological group. Then $L^{0}(\mu,G)$ is inert. Moreover, if $\mu$ is a measure (or, more generally, non-elliptic in the sense of~\cite{SchneiderSolecki21}), then $L^{0}(\mu,G)$ is whirly.\footnote{The whirliness of $L^{0}(\mu,G)$ has been noted already in~\cite[pp.~4175--4176, paragraph after the proof of Corollary~5.8]{SchneiderSolecki25}, but without explicit proof.}
	\item\label{theorem:examples.2} If $M$ is a $\mathrm{II}_{1}$ factor, then $\U(M)$ is whirly and inert.
	\item\label{theorem:examples.3} If $E$ is a countable measure-preserving equivalence relation on a non-atomic standard probability space, then $[E]$ is whirly and inert.
\end{enumerate} \end{thm}

The two theorems above entail the following corollary, which answers the question raised in Problem~\ref{problem:pestov} in the negative.

\begin{cor}\label{corollary:main} Let $(\Omega,\mu)$ be a non-atomic standard probability space. Then there exists a non-amenable closed topological subgroup $G \leq \Aut(\Omega,\mu)$ such that $G \curvearrowright (\Omega,\mu)$ is whirly. \end{cor}

This paper is organized as follows. After addressing some general matters of completeness in Section~\ref{section:completeness}, we proceed to proving Theorem~\ref{theorem:construction} in Section~\ref{section:gaussian.near.actions}. The subsequent Section~\ref{section:whirliness} revolves around characterizations and persistence theorems concerning whirliness, while Section~\ref{section:inertness} provides the corresponding results about inertness. In Section~\ref{section:examples}, we turn to the concrete examples of topological groups addressed in Theorem~\ref{theorem:examples}, providing definitions and some relevant dynamical background. The final Section~\ref{section:final} is devoted to the proofs of Theorem~\ref{theorem:examples} and Corollary~\ref{corollary:main}.

\section{Completeness}\label{section:completeness}

For the sake of convenience, we provide some preparatory remarks about completeness in this brief preliminary section.

A topological group $G$ is called \emph{Ra\u{\i}kov complete} if $G$ is complete\footnote{See~\cite[II, \S3.3]{bourbaki1} for the notion of completeness of a uniform space.} with respect to its \emph{two-sided uniformity} \begin{displaymath}
	\{ E \subseteq G \times G \mid \exists U \in \Neigh(G) \, \forall x,y \in G \colon \, x \in yU \cap Uy \Longrightarrow \, (x,y) \in E \} ,
\end{displaymath} where $\Neigh(G)$ denotes the neighborhood filter at the neutral element in $G$. The \emph{Hausdorff completion}\footnote{See~\cite[II, \S3.7]{bourbaki1} or~\cite{robertson} for the Hausdorff completion of a uniform space.} $\widehat{G}$ of a topological group $G$ with respect to its two-sided uniformity naturally constitutes a Ra\u{\i}kov complete Hausdorff topological group, which is called the \emph{Ra\u{\i}kov completion} of~$G$ (see~\cite[B8]{stroppel} for details on this construction and its properties). A Hausdorff topological group $G$ is called \emph{absolutely closed} if every embedding of $G$ into another Hausdorff topological group has closed range.

\begin{thm}[\cite{raikov}]\label{theorem:raikov} A Hausdorff topological group is Ra\u{\i}kov complete if and only if it is absolutely closed. \end{thm}

\begin{remark}[{\cite[Section~2, first paragraph, p.~4528]{PestovUspenskij}}]\label{remark:pestov.uspenskij} A second-countable Hausdorff topological group is Ra\u{\i}kov complete if and only if it is Polish (equivalently, completely metrizable). \end{remark}

\begin{remark}[{\cite[Theorem~6]{arens}}, see also~\cite{hulanicki}]\label{remark:raikov.complete} If $X$ is a compact Hausdorff space, then the homeomorphism group $\Homeo(X)$, equipped with the compact-open topology, is Ra\u{\i}kov complete. \end{remark}

%The following fact is well known. We reproduce the simple argument for the lack of a convenient reference.

%\begin{remark}\label{remark:raikov.complete} Let $X$ be a compact Hausdorff space. Then the homeomorphism group $\Homeo(X)$, equipped with the compact-open topology, is Ra\u{\i}kov complete. To see this, note that the uniformity of uniform convergence on $\Cont(X,X)$ is complete according to~\cite[II, \S4.1, Theorem~1]{bourbaki1} and~\cite[X, \S1.6, Corollary~1 on p.~282]{bourbaki2} and induces the compact-open topology by~\cite[X, \S3.4, Theorem~2]{bourbaki1}. Moreover, \begin{displaymath}
%	\left. H \, \defeq \, \left\{ (f,g) \in \Cont(X,X)^{2} \, \right\vert f \circ g = \id_{X} = g \circ f \right\}
%\end{displaymath} is closed in the product space $\Cont(X,X)^{2}$ thanks to~\cite[X, \S3.4, Proposition~9]{bourbaki2} and hence complete with respect to the induced uniformity by~\cite[II, \S3.5, Proposition~10]{bourbaki1} and~\cite[II, \S3.4, Proposition~8]{bourbaki1}. One readily checks that $\Homeo(X) \to H, \, g \mapsto (g,g^{-1})$ is an isomorphism of uniform spaces with respect to the two-sided uniformity on $\Homeo(X)$, which now implies the claim. \end{remark}

\section{Gaussian near-actions}\label{section:gaussian.near.actions}

This section contains the proof of Theorem~\ref{theorem:construction}, which is based on the standard constructions of Gaussian realizations of separable real Hilbert spaces (Theorem~\ref{theorem:gaussian.realization.existence}) and Gaussian near-actions associated to orthogonal representations (Theorem~\ref{theorem:gaussian.realization.dynamics}). Our presentation of the required background material essentially follows~\cite[Appendix~E]{KerrLiBook}. The reader is referred to~\cite[Chapter~3, Sections~11, 12]{GlasnerBook}, \cite[Appendices~D, E]{KechrisBook}, and~\cite[Section~2.4]{PestovBook2} for alternative accounts.

We start off by clarifying some relevant terminology. A near-action $G \curvearrowright (\Omega,\mu)$ of a group $G$ on a probability space $(\Omega,\mu)$ is called \begin{itemize}
	\item[---\,] \emph{ergodic} if $\mu(A) \in \{ 0,1 \}$ for every $G$-invariant measurable subset $A \subseteq \Omega$,
	\item[---\,] \emph{weakly mixing} if the induced diagonal near-action $G \curvearrowright (\Omega \times \Omega,\mu \otimes \mu)$ is ergodic.
\end{itemize} A \emph{unitary representation} (resp., \emph{orthogonal representation}) of a group $G$ on a complex (resp., real) Hilbert space $H$ is a homomorphism from $G$ to the unitary group $\U(H)$ (resp., orthogonal group $\OG(H)$). A unitary (resp., orthogonal) representation $\pi$ of a group $G$ on a complex (resp., real) Hilbert space $H$ will be called \emph{weakly mixing} if $H$ does not contain any non-zero $\pi$-invariant finite-dimensional linear subspace (cf.~\cite[Theorem~2.23, pp.~30--31]{KerrLiBook}). For a complex (resp., real) Hilbert space $H$, the group $\U(H)$ (resp., $\OG(H)$), equipped with the strong operator topology, constitutes a topological group. A topological group $G$ is said to be \emph{unitarily representable} if $G$ is isomorphic to a topological subgroup of $\U(H)$ for some complex Hilbert space $H$, or equivalently, to a topological subgroup of $\OG(H)$ for some real Hilbert space $H$.

\begin{definition}[cf.~{\cite[Definition~E.3\,+\,E.4]{KerrLiBook}}] Let $(\Omega,\mu)$ be a probability space. A real-valued random variable $f$ on $(\Omega,\mu)$ is called \emph{centered Gaussian} if it is constant or $f_{\ast}(\mu)$ has Lebesgue density \begin{displaymath}
	\R \, \longrightarrow \, \R_{>0} , \quad x \, \longmapsto \, \frac{1}{\sqrt{2\pi\sigma^{2}}} \exp\!\left( \frac{-x^{2}}{2\sigma^{2}} \right)
\end{displaymath} for some $\sigma \in \R_{>0}$. A closed linear subspace of $L^{2}(\Omega,\mu;\R)$ is called \emph{Gaussian} if it consists entirely of centered Gaussian random variables. \end{definition}

\begin{definition} A \emph{Gaussian realization} of a separable real Hilbert space $H$ is a tuple $(\Omega,\mu,\iota)$ consisting of a standard probability space $(\Omega,\mu)$ and an isometric linear embedding $\iota \colon H \to L^{2}(\Omega,\mu;\R)$ such that the image $\iota(H)$ is Gaussian and generates the $\sigma$-algebra of $\Omega$ modulo $\mu$-null sets (cf.~\cite[Definition~E.9, p.~409]{KerrLiBook}). \end{definition}

\begin{thm}[{\cite[Proposition~E.10, p.~409]{KerrLiBook}}]\label{theorem:gaussian.realization.existence} Every separable real Hilbert space admits a Gaussian realization. \end{thm}

Let us recall from~\cite[Appendix~E.2]{KerrLiBook} that, if $H$ is a real Hilbert space, then the \emph{symmetric Fock space} of $H$ is defined as \begin{displaymath}
	\SF(H) \, \defeq \, \bigoplus\nolimits_{n \in \N} H^{\odot n} ,
\end{displaymath} i.e., the Hilbert space direct sum of the symmetric tensor powers of $H$, and the map \begin{displaymath}
	\OG(H) \, \longrightarrow \, \OG(\SF(H)), \quad u \, \longmapsto \, \SF(u) \defeq \bigoplus\nolimits_{n \in \N} u^{\odot n}
\end{displaymath} constitutes a topological group embedding.

\begin{thm}[cf.~{\cite[Appendices~E.3, E.4]{KerrLiBook}}]\label{theorem:gaussian.realization.dynamics} Let $(\Omega,\mu,\iota)$ be a Gaussian realization of a separable real Hilbert space~$H$. Then the following hold. \begin{enumerate}
	\item\label{theorem:gaussian.realization.dynamics.1} There is a unique isometric isomorphism $\phi_{\iota} \colon \SF(H) \to L^{2}(\Omega,\mu;\R)$ such that \begin{displaymath}
			\qquad \forall x \in H \colon \quad \phi_{\iota}\!\left(\sum\nolimits_{n \in \N} \tfrac{1}{\sqrt{n!}}x^{\odot n}\right)\, = \, e^{\iota(x) - \tfrac{1}{2} \E_{\mu}(\iota(x)^{2})} .
		\end{displaymath}
	\item\label{theorem:gaussian.realization.dynamics.2} For each $u \in \OG(H)$, there exists a unique element $\theta_{\iota}(u) \in \Aut(\Omega,\mu)$ such that \begin{displaymath}
			\qquad \forall f \in L^{2}(\Omega,\mu;\R) \colon \quad f \circ \theta_{\iota}(u)^{-1} \, = \, \left( {\phi_{\iota}} \circ {\SF(u)} \circ {\phi_{\iota}^{-1}} \right)\!(f) .
		\end{displaymath} The map $\theta_{\iota} \colon \OG(H) \to \Aut(\Omega,\mu), \, u \mapsto \theta_{\iota}(u)$ is a topological group embedding.
	\item\label{theorem:gaussian.realization.dynamics.3} If $\pi$ is any weakly mixing orthogonal representation of a group $G$ on $H$, then the near-action $G \curvearrowright^{\theta_{\iota} \circ \pi} (\Omega,\mu)$ is weakly mixing.
\end{enumerate} \end{thm}

\begin{proof} \ref{theorem:gaussian.realization.dynamics.1} This is established in~\cite[Theorem~E.13, p.~410]{KerrLiBook}.

\ref{theorem:gaussian.realization.dynamics.2} Existence and uniqueness of the map $\theta_{\iota}$ follow from~\cite[Theorem~E.14, p.~411]{KerrLiBook}. Using the prescribed property of $\theta_{\iota}$ and the fact that \begin{displaymath}
	\Aut(\Omega,\mu) \, \longrightarrow \, \OG(L^{2}(\Omega,\mu;\R)) , \quad T \, \longmapsto \, \left[f \mapsto f \circ T^{-1}\right]
\end{displaymath} is an embedding of topological groups, one readily checks that $\theta_{\iota} \colon \OG(H) \to \Aut(\Omega,\mu)$ is a topological group embedding (cf.~\cite[Exercício~2.33]{PestovBook2}).

\ref{theorem:gaussian.realization.dynamics.3} Let $\pi$ be any orthogonal representation of a group $G$ on $H$. It follows from~\ref{theorem:gaussian.realization.dynamics.1} and~\ref{theorem:gaussian.realization.dynamics.2} that $G \curvearrowright^{\theta_{\iota} \circ \pi} (\Omega,\mu)$ is the (up to conjugacy unique) Gaussian near-action associated to $\pi$ in the sense of~\cite[Definition~E.16, p.~412]{KerrLiBook}. Hence, if $\pi$ is weakly mixing, then $G \curvearrowright^{\theta_{\iota} \circ \pi} (\Omega,\mu)$ is weakly mixing by~\cite[Theorem~2.38, p.~46]{KerrLiBook}. \end{proof}

\begin{proof}[Proof of Theorem~\ref{theorem:construction}] Let $G$ be a unitarily representable, second-countable topological group. Then we find a real Hilbert space $H$ and a topological group embedding $\pi \colon G \to \OG(H)$. Since $G$ is first-countable, there exists a sequence $x \in H^{\N}$ such that the sets \begin{displaymath}
	\{ g \in G \mid \forall i \in \{ 0,\ldots,n\} \colon \, \Vert \pi(g)x_{i} - x_{i} \Vert < \epsilon \} \qquad (n \in \N, \, \epsilon \in \R_{>0})
\end{displaymath} form a neighborhood basis at the identity in $G$. As $G$ is separable, the $\pi$-invariant closed linear subspace \begin{displaymath}
	H' \, \defeq \, \overline{\operatorname{lin}\{ \pi(g)x_{n} \mid g \in G, \, n \in \N \}} \, \leq \, H
\end{displaymath} is separable, too. By construction, $\pi' \colon G \to \OG(H'), \, g \mapsto \pi(g)\vert_{H'}$ is a topological group embedding. Hence, upon replacing $\pi$ by $\pi'$, we may and will assume without loss of generality that $H$ is separable.

Let $\mathcal{K}$ denote the set of all $\pi$-invariant finite-dimensional linear subspaces of $H$ and note that $(\mathcal{K},{\subseteq})$ is directed. Since the closed linear subspaces \begin{displaymath}
	H_{0} \, \defeq \, \overline{\bigcup \mathcal{K}} \, = \, \overline{\sum \mathcal{K}} \, \leq \, H, \qquad H_{1} \, \defeq \, H_{0}^{\perp}\! \, = \, \left( \bigcup \mathcal{K} \right)^{\perp}\! \, = \, \bigcap\nolimits_{K \in \mathcal{K}} K^{\perp}\! \, \leq \, H
\end{displaymath} are $\pi$-invariant, we may consider the induced continuous homomorphisms \begin{displaymath}
	\pi_{i} \colon \, G \, \longrightarrow \, \OG(H_{i}), \quad g \, \longmapsto \, \pi(g)\vert_{H_{i}} \qquad (i \in \{ 0,1 \}) .
\end{displaymath} We proceed by inspecting $\pi_{0}$ and $\pi_{1}$ separately.

Consider the topological subgroup $\pi_{0}(G) \leq \OG(H_{0})$. As the map \begin{displaymath}
	\pi_{0}(G) \, \longrightarrow \, \prod\nolimits_{K \in \mathcal{K}} \OG(K), \quad u \, \longmapsto \, (u\vert_{K})_{K \in \mathcal{K}}
\end{displaymath} constitutes an embedding of topological groups and $\prod\nolimits_{K \in \mathcal{K}} \OG(K)$ is compact, we infer that $\pi_{0}(G)$ is precompact by~\cite[Proposition~3.7.4, p.~194]{AT}, which readily entails precompactness of the topological subgroup $\Omega_{0} \defeq \overline{\pi_{0}(G)} \leq \OG(H_{0})$ thanks to~\cite[Corollary~3.7.6, p.~195]{AT}. Since $\OG(H_{0})$ is Ra\u{\i}kov complete (see, e.g.,~\cite[Remark~B.2]{SchneiderSolecki25}), it follows that $\Omega_{0}$ is Ra\u{\i}kov complete and hence compact according to~\cite[Theorem~3.7.15, p.~197]{AT}. Let $\mu_{0}$ denote the Haar measure on $\Omega_{0}$. Note that the group embedding \begin{displaymath}
	\lambda \colon \, \Omega_{0} \, \longrightarrow \, \Aut(\Omega_{0},\mu_{0}) , \quad x \, \longmapsto \, [y \mapsto xy]
\end{displaymath} is continuous: indeed, for any measurable subset $B \subseteq \Omega_{0}$ and any $\epsilon \in \R_{>0}$, regularity of $\mu_{0}$ asserts the existence of a compact subset $C \subseteq \Omega_{0}$ and an open subset $V \subseteq \Omega_{0}$ such that $C \subseteq B \subseteq V$ and $\mu_{0}(V \setminus C) \leq \tfrac{\epsilon}{4}$, and then we find some open $U \in \Neigh(\Omega_{0})$ with $U = U^{-1}$ and $UC \subseteq V$, whence
\begin{align*}
	d_{\mu}(xC,C) \, &= \, \mu_{0}(xC \triangle C) \, = \, \mu_{0}((xC) \setminus C) + \mu_{0}(C \setminus xC) \\
	& = \, \mu_{0}((xC) \setminus C) + \mu_{0}\!\left(\left(x^{-1}C\right) \setminus C\right) \, \leq \, 2\mu_{0}(V\setminus C) \, \leq \, \tfrac{\epsilon}{2}
\end{align*} and thus \begin{displaymath}
	d_{\mu}(xB,B) \, \leq \, d_{\mu}(xB,xC) + d_{\mu}(xC,C) + d_{\mu}(C,B) \, \leq \, \epsilon
\end{displaymath} for all $x \in U$. Since $\Omega_{0}$ is compact and $\Aut(\Omega_{0},\mu_{0})$ is Hausdorff, $\lambda$ is a topological embedding by~\cite[I, \S9.4, Corollary~2]{bourbaki1}. Now, consider the induced continuous homomorphism \begin{displaymath}
	\iota_{0} \defeq \lambda \circ {\pi_{0}} \colon \, G \, \longrightarrow \, \Aut(\Omega_{0},\mu_{0}) .
\end{displaymath} Of course, $\emptyset$ and $\Omega_{0}$ are the only $\Omega_{0}$-left-invariant subsets of $\Omega_{0}$. In particular, $\Omega_{0} \curvearrowright^{\lambda} (\Omega_{0},\mu_{0})$ is ergodic. Since $\lambda$ is continuous and $\pi_{0}(G)$ is dense in $\Omega_{0}$, it follows that $\pi_{0}(G) \curvearrowright^{\lambda} (\Omega_{0},\mu_{0})$ is ergodic, i.e., $G \curvearrowright^{\iota_{0}} (\Omega_{0},\mu_{0})$ is ergodic.

By definition, $\pi_{1}$ is weakly mixing. Thanks to Theorem~\ref{theorem:gaussian.realization.existence} and Theorem~\ref{theorem:gaussian.realization.dynamics}, there exist a standard probability space $(\Omega_{1},\mu_{1})$ and a topological group embedding $\theta \colon \OG(H_{1}) \to \Aut(\Omega_{1},\mu_{1})$ such that the near-action $G \curvearrowright^{\iota_{1}} (\Omega_{1},\mu_{1})$ induced by $\iota_{1} \defeq \theta \circ {\pi_{1}} \colon G \to \Aut(\Omega_{1},\mu_{1})$ is weakly mixing.

Note that $(\Omega,\mu) \defeq (\Omega_{0} \times \Omega_{1}, \mu_{0} \otimes \mu_{1})$ is a standard probability space. As both $\lambda$ and $\theta$ are topological group embeddings, it follows that the map \begin{displaymath}
	\Omega_{0} \times \OG(H_{1}) \, \longrightarrow \, \Aut(\Omega,\mu), \quad (g,h) \, \longmapsto \, \lambda(g) \otimes \theta(h)
\end{displaymath} is an embedding of topological groups. Moreover, due to $\pi$ being a topological group embedding, \begin{displaymath}
	G \, \longrightarrow \, \Omega_{0} \times \OG(H_{1}) , \quad g \, \longmapsto \, (\pi_{0}(g),\pi_{1}(g))
\end{displaymath} constitutes an embedding of topological groups. Hence, the composition \begin{displaymath}
	\iota \colon \, G \, \longrightarrow \, \Aut(\Omega,\mu), \quad g \, \longmapsto \, \iota_{0}(g) \otimes \iota_{1}(g) = \lambda(\pi_{0}(g)) \otimes \theta(\pi_{1}(g))
\end{displaymath} is a topological group embedding.

Finally, since $G \curvearrowright^{\iota_{0}} (\Omega_{0},\mu_{0})$ is ergodic and $G \curvearrowright^{\iota_{1}} (\Omega_{1},\mu_{1})$ is weakly mixing, $G \curvearrowright^{\iota} (\Omega,\mu)$ is ergodic by~\cite[Theorem~2.25, p.~33]{KerrLiBook}\footnote{The standing assumption of~\cite{KerrLiBook} requiring countability of the acting group is not used in the proof.}. \end{proof}

\section{Whirliness}\label{section:whirliness}

A topological group $G$ is called \emph{whirly}~\cite[Remark~5.7]{SchneiderSolecki25} if, for every continuous action of $G$ on a compact Hausdorff space $X$ and every $G$-invariant regular Borel probability measure $\mu$ on $X$, the \emph{support} of $\mu$, i.e.,\begin{displaymath}
	\spt (\mu) \, \defeq \, \{ x \in X \mid \forall U \subseteq X \text{ open}\colon \, x \in U \Longrightarrow \mu(U) > 0 \} ,
\end{displaymath} is contained in the set \begin{displaymath}
	\Fix_{G}(X) \, \defeq \, \{ x \in X \mid \forall g \in G \colon \, gx = x \}
\end{displaymath} of $G$-fixed points in $X$. A topological group is called \emph{whirly amenable}~\cite{pestov10} if it is whirly and amenable. These concepts have their origins in~\cite{GlasnerTsirelsonWeiss,GlasnerWeiss}. Whirly amenability implies extreme amenability~\cite[Corollary~3.4]{pestov10} (see also~\cite[Remark~5.7]{SchneiderSolecki25}).

\begin{remark} A topological group is called \emph{exotic}~\cite{HererChristensen} if it does not have any non-trivial continuous unitary representations (see~\cite{HererChristensen,banaszczyk,banaszczyk2,megrel,Pestov07,CarderiThom,SchneiderSolecki25,SchneiderThom} for examples). While exoticness implies whirliness~\cite[Remark~1.7]{GlasnerTsirelsonWeiss} (see also~\cite[Example~3.3]{pestov10} or~\cite[Remark~5.7]{SchneiderSolecki25}), the automorphism group of a probability space is unitarily representable and thus does not contain any non-trivial exotic topological subgroup. Hence, exotic groups play no part in the solution of Problem~\ref{problem:pestov}. \end{remark}

We collect some characterizations and persistence properties of the class of whirly topological groups.

\begin{lem}[cf.~{\cite[Theorem~3.1]{pestov10}}]\label{lemma:whirly.separable} A separable topological group $G$ is whirly if and only if $\spt(\mu) \subseteq \Fix_{G}(X)$ for every continuous action of $G$ on a metrizable compact space $X$ and every $G$-invariant Borel probability measure $\mu$ on $X$. \end{lem}

\begin{proof} ($\Longrightarrow$) Since every Borel probability measure on a Polish space, such as a metrizable compact space, is necessarily regular (see, e.g.,~\cite[II, Theorem~3.2, p.~29]{ParthasarathyBook}), this implication is trivial.

($\Longleftarrow$) We adapt the argument given in~\cite[Proof of Theorem~3.1,\,(2)$\Rightarrow$(4)]{pestov10}. Let $G$ be a separable topological group satisfying the right-hand side. To deduce whirliness, suppose that $G$ acts continuously on a compact Hausdorff space $X$ and let $\mu$ be a $G$-invariant regular Borel probability measure on $X$. Thanks to Urysohn's lemma, it suffices to show that \begin{equation}\label{eq:urysohn}
	\forall f \in \Cont(X) \ \forall x \in \spt(\mu) \ \forall g \in G \colon \qquad f(gx) \, = \, f(x) .
\end{equation} To this end, let $f \in \Cont(X)$. Note that $G$ admits a continuous action by automorphisms on the unital $\Cstar$-algebra $\Cont(X)$ given by \begin{displaymath}
	{}_{g}h \, \defeq \, h \circ {\alpha_{g^{-1}}} \qquad (g \in G, \, h \in \Cont(X)) ,
\end{displaymath} where $\alpha_{g} \colon X \to X, \, x \mapsto gx$ for each $g \in G$. Since $G$ is separable, $\{ {}_{g}f \mid g \in G \}$ is $\Vert \cdot \Vert_{\infty}$-separable and thus generates a separable unital $\Cstar$-subalgebra $A \leq \Cont(X)$. Moreover, $A$ is $G$-invariant. In turn, we obtain a continuous action of $G$ on the Gelfand spectrum $Y \defeq \Hom(A,\C)$ defined by \begin{displaymath}
	(gy)(h) \, \defeq \, y\!\left({}_{g^{-1}}h\right) \qquad (g \in G, \, y \in Y, \, h \in A)
\end{displaymath} as well as a $G$-equivariant continuous map \begin{displaymath}
	\pi \colon \, X \, \longrightarrow \, Y, \quad x \, \longmapsto \, [h \mapsto h(x)] .
\end{displaymath} Consequently, the push-forward measure $\pi_{\ast}(\mu)$ constitutes a $G$-invariant (regular) Borel probability measure on $Y$. From separability of $A$, we infer that the compact Hausdorff space $Y$ is metrizable. Therefore, if $x \in \spt(\mu)$ and thus $\pi(x) \in \spt (\pi_{\ast}(\mu))$, then our hypothesis asserts that $g\pi(x) = \pi(x)$ for every $g \in G$, whence \begin{displaymath}
	f(gx) \, = \, {{}_{g^{-1}}f}(x) \, = \, \pi(x)\!\left({}_{g^{-1}}f\right) \, = \, (g\pi(x))(f) \, = \, \pi(x)(f) \, = \, f(x)
\end{displaymath} for all $g \in G$. This proves~\eqref{eq:urysohn} and hence completes the argument. \end{proof}

\begin{prop}[{\cite[Theorem~3.1]{pestov10}}]\label{proposition:whirly.polish} A Polish group $G$ is whirly if and only if every ergodic weakly continuous near-action of $G$ on a standard probability space is whirly. \end{prop}

\begin{proof} This follows directly from Lemma~\ref{lemma:whirly.separable} and~\cite[Theorem~3.1,\,(1)$\Leftrightarrow$(2)]{pestov10}. \end{proof}

\begin{lem}\label{lemma:whirly.persistence} Let $G$ be a topological group. \begin{enumerate}
	\item\label{lemma:whirly.persistence.1} If $G = \overline{\langle \bigcup \{ H \leq G \mid H \text{ whirly} \} \rangle}$, then $G$ is whirly.
	\item\label{lemma:whirly.persistence.2} A dense topological subgroup $H \leq G$ is whirly if and only if $G$ is whirly.
	\item\label{lemma:whirly.persistence.3} Let $\pi \colon G \to H$ be a surjective continuous homomorphism onto another topological group $H$. If $G$ is whirly, then so is $H$.
	\item\label{lemma:whirly.persistence.4} Let $N \unlhd G$. If both $N$ and $G/N$ are whirly, then so is $G$.
\end{enumerate} \end{lem}

\begin{proof} \ref{lemma:whirly.persistence.1} Suppose that $G = \overline{\langle \bigcup \mathcal{H} \rangle}$ for $\mathcal{H} \defeq \{ H \leq G \mid H \text{ whirly} \}$. Now, consider any continuous action of $G$ on a compact Hausdorff space $X$ along with a $G$-invariant regular Borel probability measure $\mu$ on $X$. If $x \in \spt(\mu)$, then every member of $\mathcal{H}$ is contained in the closed subgroup $G_{x} \defeq \{ g \in G \mid gx = x \} \leq G$, whence $\overline{\langle \bigcup \mathcal{H} \rangle} \subseteq G_{x}$ and so $G_{x} = G$ by our hypothesis. This shows that $G$ is whirly.

\ref{lemma:whirly.persistence.2} ($\Longrightarrow$) This is a special case of~\ref{lemma:whirly.persistence.1}.

($\Longleftarrow$) Assume that $H$ acts continuously on a compact Hausdorff space~$X$ and let $\mu$ be a $G$-invariant regular Borel probability measure on $X$. Consider the topological group $\Homeo(X)$ carrying the compact-open topology. Since \begin{displaymath}
	\pi_{0} \colon \, H \, \longrightarrow \, \Homeo(X), \quad h \, \longmapsto \, [x \mapsto hx]
\end{displaymath} is a continuous homomorphism and $\Homeo(X)$ is Ra\u{\i}kov complete by Remark~\ref{remark:raikov.complete}, there exists a continuous homomorphism $\pi \colon G \to \Homeo(X)$ with $\pi\vert_{H} = \pi_{0}$ (see, e.g.,~\cite[Proposition~3.6.12]{AT}). In turn, $G \times X \to X, \, (g,x) \mapsto \pi(g)(x)$ is a continuous action. In order to prove that $\mu$ is $G$-invariant, it suffices to check that $\mu(gA) = \mu(A)$ for every closed subset $A \subseteq X$. For this purpose, consider any closed subset $A \subseteq X$. Now, for every $\epsilon \in \R_{>0}$, regularity of $\mu$ asserts the existence of an open subset $U \subseteq X$ such that $A \subseteq U$ and $\mu(U) \, \leq \, \mu(A) + \epsilon$, and due to density of $H$ in $G$ and continuity of the action there exists $h \in H$ such that $gA \subseteq hU$, whence \begin{displaymath}
	\mu(gA) \, \leq \, \mu(hU) \, = \, \mu(U) \, \leq \, \mu(A) + \epsilon .
\end{displaymath} Thus, $\mu(gA) = \mu(A)$. As $G$ is whirly, $\spt(\mu) \subseteq \Fix_{G}(X) \subseteq \Fix_{H}(X)$.

\ref{lemma:whirly.persistence.3} Assume that $H$ acts continuously on a compact Hausdorff space $X$ and let~$\mu$ be an $H$-invariant regular Borel probability measure on $X$. Then the continuous action $G \times X \to X, \, (g,x) \mapsto \pi(g)x$ leaves invariant the measure $\mu$, whence $\spt(\mu) \subseteq \Fix_{G}(X)$ due to whirliness of $G$. Since $\pi$ is surjective, it follows that $\spt(\mu) \subseteq \Fix_{H}(X)$.

\ref{lemma:whirly.persistence.4} Consider an continuous action of $G$ on a compact Hausdorff space $X$. Note that the closed subset $Y \defeq \Fix_{N}(X) \subseteq X$ is $G$-invariant: indeed, if $g \in G$ and $x \in Y$, then normality of $N$ in $G$ implies that $hgx = g\!\left(g^{-1}hg\right)\!x = gx$ for all $h \in N$, that is, $gx \in Y$. Consider the induced continuous action of $G$ on $Y$ along with the corresponding homomorphism $\alpha \colon G \to \Homeo(Y)$. Since $N \subseteq \ker(\alpha)$, there exists a unique continuous homomorphism $\beta \colon G/N \to \Homeo(Y)$ such that $\beta \circ \pi_{N} = \alpha$, where $\pi_{N} \colon G \to G/N, \, g \mapsto gN$. Now, let $\mu$ be any $G$-invariant regular Borel probability measure on $X$. Whirliness of $N$ yields that $\spt(\mu) \subseteq Y$. Hence, the restriction $\nu$ of $\mu$ to the Borel $\sigma$-algebra of $Y$ constitutes a regular Borel probability measure on $Y$. Moreover, $\nu$ is $G$-invariant and thus $G/N$-invariant. As $G/N$ is whirly, it follows that
\begin{displaymath}
	\spt(\mu) \, = \, \spt(\nu) \, \subseteq \, \Fix_{G/N}(Y) \, = \, \Fix_{G}(Y) \, = \, \Fix_{G}(X) . \qedhere
\end{displaymath} \end{proof}

\begin{lem}\label{lemma:whirly.products} If $(G_{i})_{i \in I}$ is a family of whirly topological groups, then $\prod_{i \in I} G_{i}$ is~whirly. \end{lem}

\begin{proof} First, we show that the direct product of any finite family of whirly topological groups is whirly. Upon induction (and due to the fact that trivial groups are whirly), it suffices to check that the direct product of two whirly topological groups is whirly. So, let $G$ and $H$ be whirly topological groups. Consider $N \defeq G\times\{1\} \unlhd G \times H$. Since the topological groups $N \cong G$ and $(G\times H)/N \cong H$ are whirly, we can apply Lemma~\ref{lemma:whirly.persistence}\ref{lemma:whirly.persistence.4} and get the whirliness of $G\times H$.

Turning to the general case, let $(G_{i})_{i \in I}$ be a family of whirly topological groups and consider $G \defeq \prod_{i \in I} G_{i}$. For every finite subset $F \subseteq I$, the topological subgroup \begin{displaymath}
	G_{F} \, \defeq \, \{ g \in G \mid \forall i \in I\setminus F \colon \, g_{i} = 1 \} \, \leq \, G
\end{displaymath} is isomorphic to $\prod_{i \in F} G_{i}$ and hence whirly by the assertion established in the preceding paragraph. Since $G = \overline{\bigcup\{ G_{F} \mid F \subseteq I \text{ finite}\}}$, thus $G$ is whirly by Lemma~\ref{lemma:whirly.persistence}\ref{lemma:whirly.persistence.1}. \end{proof}

\begin{lem}\label{lemma:open.subgroup} A whirly topological group does not have any proper open subgroups. \end{lem}

\begin{proof} Let $G$ be a whirly topological group. If $H$ is an open subgroup of $G$, then $G$ admits a continuous action on the compact Hausdorff space \begin{displaymath}
	X \, \defeq \, \{ 0,1 \}^{G/H}\! \, = \, \{ 0,1 \}^{\{ gH \mid g \in G \}}
\end{displaymath} defined by \begin{displaymath}
	(gx)(A) \, \defeq \, x\!\left(g^{-1}A\right) \qquad (g \in G, \, x \in X, \, A \in G/H) ,
\end{displaymath} and the product measure $\mu \defeq \nu^{\otimes G/H}$ arising from the measure $\nu \defeq \tfrac{1}{2}(\delta_{0} + \delta_{1})$ on $\{ 0,1 \}$ constitutes a $G$-invariant regular Borel probability measure on $X$ with $\spt(\mu) = X$, whence $\Fix_{G}(X) = X$ by whirliness of $G$, which implies that $H = G$. \end{proof}

\begin{prop}\label{proposition:whirly.not.locally.compact} Every whirly locally compact group is trivial. \end{prop}

\begin{proof} Let $G$ be a whirly locally compact group. Any compact identity neighborhood in $G$ generates a $\sigma$-compact, open subgroup of $G$, which necessarily coincides with $G$ by Lemma~\ref{lemma:open.subgroup}. Therefore, $G$ is $\sigma$-compact. Now, for contradiction, suppose that $G$ is non-trivial. Pick any identity neighborhood $U$ in $G$ with $U \ne G$. By the Kakutani--Kodaira theorem~\cite[Satz~6]{KakutaniKodaira} (see also~\cite[Theorem~A.9]{GreenleafMoskowitz} or~\cite[Theorem~8.7]{HewittRoss}), there exists a compact normal subgroup $N \unlhd G$ such that $N \subseteq U$ and $G/N$ is second-countable. Since locally compact groups are unitarily representable, we may apply Theorem~\ref{theorem:construction} to obtain the existence of a standard probability space $(\Omega,\mu)$ and a topological group embedding $\iota \colon G/N \to \Aut(\Omega,\mu)$ such that the induced near-action $G/N \curvearrowright^{\iota} (\Omega,\mu)$ is ergodic. Then $G/N \curvearrowright^{\iota} (\Omega,\mu)$ is whirly by Lemma~\ref{lemma:whirly.persistence}\ref{lemma:whirly.persistence.3} and Proposition~\ref{proposition:whirly.polish} and hence trivial by~\cite[Theorem~0.3(a)]{GlasnerTsirelsonWeiss} and~\cite[Proposition~3.3(b)]{GlasnerTsirelsonWeiss}. This means that $\iota(G/N)$ is trivial and thus $G/N$ is trivial. In turn, $G = N \subseteq U \subseteq G$ and so $U = G$, which constitutes the intended contradiction. \end{proof}

A topological group $G$ is called \emph{minimally almost periodic} if every continuous homomorphism from $G$ to a compact Hausdorff topological group is trivial.

\begin{cor}\label{corollary:whirly.vs.locally.compact} Every continuous homomorphism from a whirly topological group to a locally compact group is trivial. In particular, whirliness implies minimal almost periodicity. \end{cor}

\begin{proof} Let $\pi \colon G \to H$ be a continuous homomorphism from a whirly topological group $G$ to a locally compact group $H$. Then $\overline{\pi(G)} \leq H$ is whirly due to Lemma~\ref{lemma:whirly.persistence}\ref{lemma:whirly.persistence.2}$+$\ref{lemma:whirly.persistence.3} and locally compact, hence trivial by Proposition~\ref{proposition:whirly.not.locally.compact}. This means that $\pi$ is trivial. \end{proof}

%\begin{lem}\label{lemma:whirly.implies.minap} Every whirly topological group is minimally almost periodic. \end{lem}

%\begin{proof} First of all, any whirly compact Hausdorff topological group is trivial: indeed, since the Haar measure $\mu$ on a compact Hausdorff topological group $K$ is regular and invariant under the continuous action $\alpha \colon K \times K \to K, \, (x,y) \mapsto xy$, whirliness of $K$ implies that every element of $K = \spt(\mu)$ is fixed under $\alpha$, which readily entails triviality of $K$. Now, if $\pi \colon G \to H$ is a continuous homomorphism from a whirly topological group $G$ to a compact Hausdorff topological group $H$, then $\overline{\pi(G)} \leq H$ is compact, but also whirly due to Lemma~\ref{lemma:whirly.persistence}\ref{lemma:whirly.persistence.2}$+$\ref{lemma:whirly.persistence.3}, hence trivial. \end{proof}

\section{Inertness}\label{section:inertness}

A topological group $G$ is said to be \emph{inert}~\cite{SchneiderGAFA} if, for every continuous action of $G$ on a non-empty compact Hausdorff space $X$, every element of $G$ admits a fixed point in $X$, that is, \begin{displaymath}
	\forall g \in G \ \exists x \in X \colon \qquad gx \, = \, x .
\end{displaymath} This section explores the persistence of inertness under some standard constructions of topological groups.

\begin{lem}\label{lemma:inert.persistence} Let $G$ be a topological group. \begin{enumerate}
	\item\label{lemma:inert.persistence.1} If $G = \overline{\bigcup \{ H \leq G \mid H \text{ inert} \}}$, then $G$ is inert.
	\item\label{lemma:inert.persistence.2} A dense topological subgroup $H \leq G$ is inert if and only if $G$ is inert.
	\item\label{lemma:inert.persistence.3} Let $\pi \colon G \to H$ be a surjective continuous homomorphism onto another topological group $H$. If $G$ is inert, then so is $H$.
\end{enumerate} \end{lem}

\begin{proof} \ref{lemma:inert.persistence.1} This is a straightforward adaptation of the argument given in~\cite[Proof of Lemma~11.13(2)]{SchneiderGAFA}. Suppose that the set $S \defeq \bigcup \{ H \leq G \mid H \text{ inert} \}$ is dense in $G$. Now, if $G$ acts continuously on a non-empty compact Hausdorff space $X$, then \begin{displaymath}
	S \, \subseteq \, \{g \in G \mid \exists x \in X \colon \, gx = x \} \, \eqdef \, T ,
\end{displaymath} wherefore $\overline{T} = G$ and hence $T = G$ by~\cite[Lemma~11.13(1)]{SchneiderGAFA}. Thus, $G$ is inert.

\ref{lemma:inert.persistence.2} ($\Longrightarrow$) This is a special case of~\ref{lemma:inert.persistence.1}.

($\Longleftarrow$) Assume that $H$ acts continuously on a non-empty compact Hausdorff space~$X$. Consider the topological group $\Homeo(X)$ carrying the compact-open topology. Since \begin{displaymath}
	\pi_{0} \colon \, H \, \longrightarrow \, \Homeo(X), \quad h \, \longmapsto \, [x \mapsto hx]
\end{displaymath} is a continuous homomorphism and $\Homeo(X)$ is Ra\u{\i}kov complete by Remark~\ref{remark:raikov.complete}, there exists a continuous homomorphism $\pi \colon G \to \Homeo(X)$ with $\pi\vert_{H} = \pi_{0}$ (see, e.g.,~\cite[Proposition~3.6.12]{AT}). In turn, $G \times X \to X, \, (g,x) \mapsto \pi(g)(x)$ is a continuous action, and so inertness of $G$ entails that \begin{displaymath}
	\forall g \in G \ \exists x \in X \colon \qquad \pi(g)(x) \, = \, x .
\end{displaymath} As $\pi\vert_{H} = \pi_{0}$, we conclude that every element of $H$ admits a fixed point in $X$.

\ref{lemma:inert.persistence.3} Assume that $H$ acts continuously on a non-empty compact Hausdorff space $X$. Then $G \times X \to X, \, (g,x) \mapsto \pi(g)x$ is a continuous action, wherefore inertness of $G$ asserts that every element of $G$ admits a fixed point in $X$. Since $\pi$ is surjective, it follows that every element of $H$ has a fixed point in $X$. \end{proof}

\begin{lem}\label{lemma:inert.products} If $(G_{i})_{i \in I}$ is a family of inert topological groups, then $\prod_{i \in I} G_{i}$ is inert. \end{lem}

\begin{proof} We begin by showing that the direct product of any finite family of inert topological groups is inert. By induction, it suffices to check that the direct product of two inert topological groups is inert. Let $G$ and $H$ be inert topological groups. Consider any continuous action of $G \times H$ on a non-empty compact Hausdorff space~$X$. Let $(g,h) \in G \times H$. Since $G \times \{ 1 \}$ is inert, \begin{displaymath}
	Y \, \defeq \, \{ x \in X \mid (g,1)x = x\}
\end{displaymath} is non-empty. Moreover, one readily checks that $Y$ is a $\{1\} \times H$-invariant and closed subset of $X$. As $\{ 1\} \times H$ is inert, thus there exists $x \in Y$ such that $(1,h)x = x$. We conclude that $(g,h)x = (g,1)(1,h)x = (g,1)x = x$, as desired.

For the proof of the general statement, let $G \defeq \prod_{i \in I} G_{i}$ for any family $(G_{i})_{i \in I}$ of inert topological groups. For every finite subset $F \subseteq I$, the topological subgroup \begin{displaymath}
	G_{F} \, \defeq \, \{ g \in G \mid \forall i \in I\setminus F \colon \, g_{i} = 1 \} \, \leq \, G
\end{displaymath} is isomorphic to $\prod_{i \in F} G_{i}$ and hence inert by the assertion established above. Since $\bigcup\{ G_{F} \mid F \subseteq I \text{ finite}\}$ is dense in $G$, thus $G$ is inert thanks to Lemma~\ref{lemma:inert.persistence}\ref{lemma:inert.persistence.1}. \end{proof}

\begin{prop}[{\cite[p.~1612, paragraph before Corollary~1.6]{SchneiderGAFA}}] Every continuous homomorphism from an inert topological group to a locally compact group is trivial. In particular, inertness implies minimal almost periodicity. \end{prop}

\begin{proof} Let $\pi \colon G \to H$ be a continuous homomorphism from an inert topological group $G$ to a locally compact group $H$. According to~\cite[Theorem~2.2.1]{veech}, the topological group $H$ admits a free continuous action on a non-empty compact Hausdorff space $X$. In turn, $G \times X \to X, \, (g,x) \mapsto \pi(g)x$ is a continuous action. Consequently, for every $g \in G$, inertness of $G$ asserts the existence of $x \in X$ such that $\pi(g)x = x$, which entails that $\pi(g) = 1$ by freeness of the inital action. This shows that $\pi$ is trivial. \end{proof}

\section{Examples}\label{section:examples}

Our Theorem~\ref{theorem:examples} establishes whirliness and inertness for three concrete families of topological groups. The current section's purpose is to briefly introduce these examples along with some of their relevant properties.

\subsection{$L^{0}$ groups}\label{subsection:L0.groups} Let $\mu$ be a \emph{submeasure} on a Boolean algebra $\mathcal{A}$, i.e., a function $\mu \colon \mathcal{A} \to \R$ such that \begin{itemize}
	\item[---\,] $\mu(0) = 0$,
	\item[---\,] $\mu(A) \leq \mu(B)$ for all $A,B \in \mathcal{A}$ with $A \leq B$,
	\item[---\,] $\mu(A \vee B) \leq \mu(A) + \mu(B)$ for all $A,B \in \mathcal{A}$.
\end{itemize} Then $\mu$ is said to be a \emph{measure} if $\mu(A \vee B) = \mu(A) + \mu(B)$ for any two $A,B \in \mathcal{A}$ with $A \wedge B = 0$. Moreover, $\mu$ is called \emph{diffuse} if, for every $\epsilon \in \R_{>0}$, there exists a finite subset $\mathcal{Q} \subseteq \mathcal{A}$ such that $1 = \bigvee \mathcal{Q}$ and $\sup_{Q \in \mathcal{Q}} \mu(Q) \leq \epsilon$.

Let $\mathcal{A}$ be a Boolean algebra and let $G$ be a group. A \emph{finite $G$-partition of unity} in $\mathcal{A}$ is a map $a \colon G \to \mathcal{A}$ such that $\{g \in G \mid a(g) \neq 0\}$ is finite and $\bigveedot_{g \in G} a(g) = 1$. The set $S(\mathcal{A},G)$ of all finite $G$-partitions of unity in $\mathcal{A}$ equipped with the multiplication
\begin{displaymath}
    (ab)(g) \, \defeq \, \bigvee \{ a(x) \wedge b(y) \mid x,y \in G, \, xy = g \} \qquad (a, b \in S(\mathcal{A},G),\, g \in G)
\end{displaymath} forms a group. Now, suppose that $\mu$ is a submeasure on $\mathcal{A}$ and that $G$ is a topological group. Then the family of sets
\begin{displaymath}
    N_{\mu}(U,\varepsilon) \, \defeq \, \left\{ a \in S(\mathcal{A},G) \left\vert \, \mu\!\left(\bigvee\nolimits_{g\in U} a(g)\right) \leq \varepsilon \right\}\!\right. \qquad (U \in \Neigh(G),\, \varepsilon \in \R_{>0})
\end{displaymath} constitutes a neighborhood basis at the neutral element for a unique group topology on $S(\mathcal{A},G)$ (cf.~\cite[Section~2]{SchneiderSolecki25}). We denote by $S(\mu,G)$ the resulting topological group and define \begin{displaymath}
	L^{0}(\mu,G) \, \defeq \, \widehat{S(\mu,G)} ,
\end{displaymath} i.e., the topological group $L^{0}(\mu,G)$ is the Ra\u{\i}kov completion of $S(\mu,G)$ (see Section~\ref{section:completeness} and~\cite[Definition~2.2]{SchneiderSolecki25}).

\begin{remark}\label{remark:L0.identification} Let $(\Omega,\mathcal{B},\mu)$ be a finite measure space and let $G$ be a Polish group. Then the set $L^{0}(\Omega,\mathcal{B},\mu;G)$ of all $\mu$-equivalence classes of $\mathcal{B}$-measurable functions from $\Omega$ to $G$, equipped with pointwise multiplication and the topology of convergence in $\mu$, which is generated by the metric \begin{displaymath}
    d^{0}_{\mu}(f,g) \, \defeq \, \inf \{ \epsilon \in \R_{>0} \mid \mu(\{x \in \Omega \mid d(f(x),g(x)) > \epsilon \}) \leq \epsilon \}
\end{displaymath} for any metric $d$ generating the topology of $G$, constitutes a topological group (see~\cite[pp.~5--6]{moore} and~\cite[\S2]{SchneiderSolecki25}). Moreover, $L^{0}(\mu,G) \cong L^{0}(\Omega,\mathcal{B},\mu;G)$~\cite[Remark~2.3]{SchneiderSolecki25}. If $\mathcal{B}$ is countably generated, then $L^{0}(\Omega,\mathcal{B},\mu;G)$ is Polish by~\cite[Proposition~7]{moore}. Moreover, if $\mu$ is a probability measure, $G$ is discrete and $d$ is the $\{ 0,1\}$-valued metric on $G$, then \begin{displaymath}
    d_{\mu}^{0}(f,g) \, = \, \mu(\{\omega \in \Omega \mid f(\omega) \neq g(\omega)\})
\end{displaymath} for all $f,g \in L^{0}(\Omega,\mathcal{B},\mu;G)$. \end{remark}

We will need the following basic fact about unitary representability of $L^{0}$ groups. The reader is referred to~\cite{solecki2014} for a classification of continuous unitary representations of $L^{0}$ groups taking values in the circle and to~\cite{solecki2023,pestov18} for applications thereof.

\begin{lem}[see, e.g.,~\cite{SchneiderSolecki25}] Let $\mu$ be a measure. If a topological group $G$ is unitarily representable, then $L^{0}(\mu,G)$ is unitarily representable. \end{lem}

\begin{proof} This follows from~\cite[Lemma~5.13(iii)\,+\,Remark~B.3\,+\,Remark~5.11]{SchneiderSolecki25}. \end{proof}

The subsequent characterization of amenability for $L^{0}$ groups generalizes earlier results from~\cite{glasner98,pestov02,pestov10}.

\begin{thm}[{\cite[Theorem~1.1]{PestovSchneider}}, see also~{\cite[Corollary~7.8]{SchneiderSolecki21}}]\label{theorem:L0.amenable} Let $G$ be a topological group and let $\mu$ be a non-zero diffuse measure. The following are equivalent. \begin{itemize}
	\item[---\,] $G$ is amenable.
	\item[---\,] $L^{0}(\mu,G)$ is amenable.
	\item[---\,] $L^{0}(\mu,G)$ is extremely amenable.
	\item[---\,] $L^{0}(\mu,G)$ is whirly amenable.
\end{itemize} \end{thm}

%\footnote{cf.~\cite[Corollary~4.5]{rickert}}

The following result by Sabok~\cite{sabok} (building on earlier work of Farah and Solecki~\cite{FarahSolecki}) applies to arbitrary diffuse submeasures, but only for abelian coefficient groups.

\begin{thm}[\cite{sabok}]\label{theorem:sabok} If $\mu$ is a diffuse submeasure, then $L^{0}(\mu,\Z)$ is extremely amenable. \end{thm}

\subsection{Unitary groups}\label{subsection:unitary.groups}

Let $M$ be a von Neumann algebra and let $M_{\ast}$ denote the canonical (and up to isomorphism unique) Banach predual of $M$ (cf.~\cite[\S1.13]{sakai}). The \emph{ultrastrong topology}, i.e., the topology generated by the seminorms of the form \begin{displaymath}
	M \, \longrightarrow \, \R_{\geq 0}, \quad x \, \longmapsto \, \phi(x^{\ast}x)^{1/2} \qquad \left(\phi \in M_{\ast}^{+}\right) \!,
\end{displaymath} and the \emph{ultraweak topology}, that is, the weak topology $\sigma (M,M_{\ast})$, coincide on the \emph{unitary group} \begin{displaymath}
	\U(M) \, \defeq \, \{ u \in M \mid u^{\ast}u = uu^{\ast} = 1 \} .
\end{displaymath} Endowed with this topology, $\U(M)$ constitutes a topological group: multiplication is continuous with respect to the ultrastrong topology, while inversion is continuous with respect to the ultraweak topology. Moreover, if $M$ is a finite factor, then there exists a (necessarily faithful, normal) unique tracial state $\tau_{M} \colon M \to \C$ (see, e.g.,~\cite[Theorem~8.2.8, p.~517]{KadisonRingrose2} or~\cite[V, \S2, Theorem~2.6, p.~312]{Takesaki1}), and one can use the isomorphism $M_{\ast} \cong L^{1}(M,\tau_{M})$ (cf.~\cite[V, \S2, Theorem~2.18, p.~321]{Takesaki1}) to show that the ultraweak topology on $\U(M)$ is generated by the bi-invariant metric \begin{displaymath}
	\U(M) \times \U(M) \, \longrightarrow \, [0,2], \quad (u,v) \, \longmapsto \, \tau_{M}((u-v)^{\ast}(u-v))^{1/2} .
\end{displaymath}

\begin{remark}\label{remark:unitary.group} Let $M$ be a von Neumann algebra. \begin{enumerate}
	\item\label{remark:unitary.group.1} Up to a ${}^{\ast}$-algebra isomorphism, $M$ may be identified with a weakly closed unital ${}^{\ast}$-subalgebra of $\B(H)$ for some complex Hilbert space $H$ (see~\cite[Theorem~1.16.7, p.~41]{sakai}). The ultraweak topology on $\U(M)$ coincides with the strong operator topology induced from $H$ (see~\cite[Proposition~1.15.2, p.~35]{sakai}). In particular, $\U(M)$ is unitarily representable. Moreover, with respect to the strong operator topology, $\U(H)$ is Ra\u{\i}kov complete (see, e.g.,~\cite[Remark~B.2]{SchneiderSolecki25}) and $\U(M) = \U(H) \cap M$ is closed in $\U(H)$, whence $\U(M)$ is Ra\u{\i}kov complete.
	\item\label{remark:unitary.group.2} Thanks to the Banach--Alaoglu theorem, the closed unit ball $B \subseteq M$ is compact with respect to the ultraweak topology. Now, if $M_{\ast}$ is separable, then $B$ is metrizable and hence separable, whence $\U(M)$ is metrizable and separable, thus Polish by~\ref{remark:unitary.group.1} and Remark~\ref{remark:pestov.uspenskij}.
\end{enumerate} \end{remark}

\begin{thm}[{\cite[Theorem~1]{paterson}}\footnote{This extends de la Harpe's result~\cite{DeLaHarpe} to von Neumann algebras with non-separable predual, as suggested by Haagerup~\cite[Remark on p.~309]{haagerup}.}]\label{theorem:paterson} A von Neumann algebra $M$ is injective if and only if the topological group $\U(M)$ is amenable. \end{thm}

For the reader's convenience, we include a proof of the following well-known fact, which is used in the proof of Theorem~\ref{theorem:examples}\ref{theorem:examples.2}.

\begin{lem}\label{lemma:masa} Let $N$ be a maximal abelian ${}^{\ast}$-subalgebra of a $\mathrm{II}_{1}$ factor $M$. Then there exist a diffuse probability space $(\Omega,\mathcal{B},\mu)$ and a ${}^{\ast}$-isomorphism $\iota \colon L^{\infty}(\Omega,\mathcal{B},\mu;\C) \to N$ such that \begin{displaymath}
	\forall B \in \mathcal{B} \colon \qquad \mu(B) \, = \, \tau_{M}(\iota(\chi_{B})) .
\end{displaymath} \end{lem}

\begin{proof} A standard argument using separate continuity of the multiplication and continuity of the ${}^{\ast}$-operation with respect to the ultraweak topology on $M$ (cf.~\cite[Theorem~1.7.8, p.~18]{sakai}) shows that $N$ is closed in $M$ with respect to the ultraweak topology, and thus $N$ itself constitutes a von Neumann algebra by~\cite[VI.6.4, Corollary on p.~282]{schaefer}. Hence, the classification of abelian von Neumann algebras (see, e.g.,~\cite[Proposition~1.18.1, p.~45]{sakai} or \cite[Corollary~7.20]{BlecherGoldsteinLabuschagne}) asserts the existence of a localizable measure space $(\Omega,\mathcal{B},\nu)$ and a ${}^{\ast}$-isomorphism $\iota \colon L^{\infty}(\Omega,\mathcal{B},\nu;\C) \to N$. We claim that the map \begin{displaymath}
	\mu \colon \, \mathcal{B} \, \longrightarrow \, [0,1] , \quad B \, \longmapsto \, \tau_{M}(\iota(\chi_{B}))
\end{displaymath} is a probability measure on $\mathcal{B}$. Evidently, $\mu(\Omega) = \tau_{M}(\iota(1)) = \tau_{M}(1) = 1$. Also, if $A,B \in \mathcal{B}$ are disjoint, then \begin{align*}
	\mu(A \cup B) \, & = \, \tau_{M}(\iota(\chi_{A \cup B})) \, = \, \tau_{M}(\iota(\chi_{A} + \chi_{B})) \, = \, \tau_{M}(\iota(\chi_{A}) + \iota(\chi_{B})) \\
		             & = \, \tau_{M}(\iota(\chi_{A})) + \tau_{M}(\iota(\chi_{B})) \, = \, \mu(A) + \mu(B) .
\end{align*} To verify $\sigma$-continuity, let $(B_{n})_{n \in \N}$ be an ascending chain in $\mathcal{B}$ and let $B \defeq \bigcup_{n \in \N} B_{n}$. Then Lebesgue's dominated convergence theorem yields that \begin{displaymath}
	\langle \chi_{B_{n}}, f \rangle \, = \, \int \chi_{B_{n}}f \, \mathrm{d}\mu \, \stackrel{n \to \infty}{\longrightarrow} \, \int \chi_{B}f \, \mathrm{d}\mu \, = \, \langle \chi_{B}, f \rangle
\end{displaymath} for every $f \in L^{1}(\Omega,\mathcal{B},\nu;\C)$. Since $L^{\infty}(\Omega,\mathcal{B},\nu;\C)_{\ast} \cong L^{1}(\Omega,\mathcal{B},\nu;\C)$, this means that $\chi_{B_{n}} \to \chi_{B}$ in the ultraweak topology of $L^{\infty}(\Omega,\mathcal{B},\nu;\C)$ and so $\iota(\chi_{B_{n}}) \to \iota(\chi_{B})$ in the ultraweak topology of $N$, which coincides with the relative ultraweak topology inherited from $M$ (cf.~\cite[Proposition~1.24.5, p.~78]{sakai}). As $\tau_{M}$ is continuous with respect to the ultraweak topology on $M$, it follows that \begin{displaymath}
	\mu(B_{n}) \, = \, \tau_{M}(\iota(\chi_{B_{n}})) \, \stackrel{n \to \infty}{\longrightarrow} \, \tau_{M}(\iota(\chi_{B})) \, = \, \mu(B) ,
\end{displaymath} as desired. This shows that $\mu$ is a probability measure on $\mathcal{B}$. From faithfulness of $\tau_{M}$, we moreover infer that, for every $B \in \mathcal{B}$, \begin{align*}
	\nu(B) = 0 \quad & \Longleftrightarrow \quad \chi_{B} = 0 \text{ in $L^{\infty}(\Omega,\mathcal{B},\nu;\C)$} \quad \Longleftrightarrow \quad \iota(\chi_{B}) = 0 \\
		             & \Longleftrightarrow \quad \tau_{M}(\iota(\chi_{B})) = 0 \quad \Longleftrightarrow \quad \mu(B) = 0 .
\end{align*} Consequently, $L^{\infty}(\Omega,\mathcal{B},\nu;\C) = L^{\infty}(\Omega,\mathcal{B},\mu;\C)$. Finally, since $M$ is diffuse, $N$ is diffuse by~\cite[Proposition~4.6]{KadisonLiu} and hence its isomorphic copy $L^{\infty}(\Omega,\mathcal{B},\mu;\C)$ is diffuse, which readily implies that $(\Omega,\mathcal{B},\mu)$ is diffuse. This completes the proof. \end{proof}

\begin{remark}\label{remark:unitary.group.vs.L0} Let $(X,\mathcal{B},\mu)$ be a finite measure space. Then the relative ultraweak topology on $\U(L^{\infty}(X,\mathcal{B},\mu;\C)) = L^{0}(\Omega,\mathcal{B},\mu;\T)$ inherited from $L^{\infty}(X,\mathcal{B},\mu;\C)$ coincides with the topology of convergence in $\mu$. \end{remark}

\enlargethispage{4mm}

\subsection{Full groups}\label{subsection:full.groups}

Let $\Omega$ be a standard Borel space. A Borel equivalence relation $E \subseteq \Omega \times \Omega$ is called \emph{countable} (resp., \emph{finite}) if every $E$-equivalence class is countable (resp., finite). By~\cite[Theorem~1]{FeldmanMoore}, every countable Borel equivalence relation arises as the orbit equivalence relation of a Borel action of a countable discrete group on $\Omega$. Let $\mu$ be a probability measure on $\Omega$. A countable Borel equivalence relation $E$ on $(\Omega,\mu)$ is called \emph{measure-preserving} if $E$ arises as the orbit equivalence relation of a countable discrete group acting on $(\Omega,\mu)$ by measure-preserving Borel automorphisms.

Let $E$ be a countable measure-preserving equivalence relation on a standard probability space $(\Omega,\mu)$. A measurable subset $A \subseteq \Omega$ is said to be \emph{$E$-invariant} if $[x]_{E} \subseteq A$ for every $x \in A$. The equivalence relation $E$ is called \emph{ergodic} if $\mu(A) \in \{ 0,1 \}$ for every measurable $E$-invariant subset $A \subseteq \Omega$. The \emph{full group} of $E$ is defined as \begin{displaymath}
	[E] \, \defeq \, \{ T \in \Aut(\Omega,\mu) \mid (\omega,T(\omega)) \in E \text{ for $\mu$-a.e.\ } \omega \in \Omega \} .
\end{displaymath} Then $[E]$, equipped with the \emph{uniform topology}, i.e., the topology generated by the bi-invariant metric \begin{displaymath}
	d_{\mu} \colon \, [E] \times [E] \, \longrightarrow \, [0,1] , \quad (S,T) \, \longmapsto \, \mu(\{ \omega \in \Omega \mid S(\omega) \ne T(\omega) \}) ,
\end{displaymath} is Polish by~\cite[Proposition~3.2]{KechrisBook}. The following remark generalizes~\cite[Remark~5.1]{GiordanoPestov}.

\enlargethispage{5mm}

\begin{remark}[\cite{FeldmanMoore,FeldmanMooreII}] Let $E$ be a countable measure-preserving equivalence relation on a standard probability space $(\Omega,\mu)$. According to~\cite[Lemma~1.12]{HoudayerNotes} (see also~\cite[Lemma~1.5.2]{AnantharamanPopa}), there is a measure $\nu$ on the measurable space $E \subseteq \Omega \times \Omega$ given by \begin{displaymath}
	\nu(B) \, \defeq \, \int \vert \{ \omega' \in \Omega \mid (\omega,\omega') \in B \} \vert \, \mathrm{d}\mu(\omega) \, = \, \int \vert \{ \omega' \in \Omega \mid (\omega',\omega) \in B \} \vert \, \mathrm{d}\mu(\omega)
\end{displaymath} for every measurable subset $B \subseteq E$. For every $a \in L^{\infty}(E,\nu)$, we obtain a bounded linear operator \begin{displaymath}
	m_{a} \colon \, L^{2}(E,\nu) \, \longrightarrow \, L^{2}(E,\nu), \quad f \, \longmapsto \, (a \circ {\pr_{1}})\cdot f ,
\end{displaymath} where $\pr_{1} \colon E \to \Omega, \, (\omega,\omega') \to \omega$. Moreover, for $T \in [E]$, consider the unitary operator \begin{displaymath}
	u_{T} \colon \, L^{2}(E,\nu) \, \longrightarrow \, L^{2}(E,\nu), \quad f \, \longmapsto \, f \circ \left( T^{-1} \!\times \id_{\Omega} \right) .
\end{displaymath} The \emph{von Neumann algebra of $E$}~\cite[\S1.5.2.]{AnantharamanPopa} (see also~\cite[Definition~1.14]{HoudayerNotes}) is defined as the bicommutant \begin{displaymath}
	\vN(E) \, \defeq \, (\{ m_{a} \mid a \in L^{\infty}(E,\nu) \} \cup \{ u_{T} \mid T \in [E] \})'' \, \subseteq \, \B\!\left(L^{2}(E,\nu)\right) .
\end{displaymath} Suppose now that $E$ is ergodic. Then, by~\cite[Proposition~1.15(2)\,+\,Proposition~1.19]{HoudayerNotes} (see also~\cite[Proposition~1.5.5(ii)]{AnantharamanPopa}), the von Neumann algebra $\vN(E)$ is a finite factor with the unique tracial state on $\vN(E)$ given by \begin{displaymath}
	\tau_{\vN(E)}(a) \, = \, \langle a(\chi_{\Delta_{\Omega}}),\chi_{\Delta_{\Omega}} \rangle \qquad (a \in \vN(E)) ,
\end{displaymath} where $\Delta_{\Omega} \defeq \{ (\omega,\omega) \mid \omega \in \Omega \} \subseteq E$. In turn, \begin{align*}
	\tau_{\vN(E)}(u_{S}^{\ast}u_{T}) \, & = \, \langle u_{T}(\chi_{\Delta_{\Omega}}),u_{S}(\chi_{\Delta_{\Omega}}) \rangle                                                                                                                     \\
		& = \, \int \chi_{\Delta_{\Omega}}\!\left(T^{-1}\omega, \omega' \right) \chi_{\Delta_{\Omega}}\!\left(S^{-1}\omega, \omega' \right) \, \mathrm{d}\nu(\omega,\omega')                                   \\
		& = \, \int \chi_{\gr(T^{-1})}(\omega, \omega') \chi_{\gr(S^{-1})}(\omega, \omega') \, \mathrm{d}\nu(\omega,\omega')                                                                                   \\
		& = \, \int \chi_{\gr(T^{-1}) \cap \gr(S^{-1})} \, \mathrm{d}\nu \, = \, \nu\!\left( \gr\!\left(T^{-1}\right) \cap \gr\!\left(S^{-1}\right)\right)                                                     \\
		& = \, \int \, \left\lvert \left\{ \omega' \in \Omega \left\vert \, (\omega',\omega) \in  \gr\!\left(T^{-1}\right) \cap \gr\!\left(S^{-1}\right) \right\} \!\right. \right\rvert \mathrm{d}\mu(\omega) \\
		& = \, \mu(\{ \omega \in \Omega \mid S(\omega) = T(\omega) \}) \, = \, 1-d_{\mu}(S,T)
\end{align*} and thus \begin{align*}
	\tau_{\vN(E)}\!\left( (u_{S}-u_{T})^{\ast}(u_{S}-u_{T}) \right) \, = \, 2-2\Re \tau_{\vN(E)}(u_{S}^{\ast}u_{T}) \, = \, 2d_{\mu}(S,T)
\end{align*} for all $S,T \in [E]$. Hence, the homomorphism $[E] \to \U(\vN(E)), \, T \mapsto u_{T}$ constitutes a topological group embedding. In particular, $[E]$ is unitarily representable. \end{remark}

Let $(\Omega,\mu)$ be a standard probability space and let $E,F$ be countable Borel equivalence relations on $\Omega$. We say that $E$ and $F$ \emph{agree almost everywhere} if there exists a measurable subset $\Omega' \subseteq \Omega$ with $\mu(\Omega') = 1$ such that \begin{displaymath}
	E \cap (\Omega' \times \Omega') \, = \, F \cap (\Omega' \times \Omega') .
\end{displaymath} A countable measure-preserving equivalence relation $E$ on $(\Omega,\mu)$ is called \emph{hyperfinite} if there exists an increasing sequence $(E_{n})_{n \in \N}$ of finite Borel equivalence relations on $\Omega$ such that $E$ agrees with $\bigcup_{n \in \N} E_{n}$ almost everywhere. According to the celebrated Connes--Feldman--Weiss theorem~\cite[Theorem 10]{ConnesFeldmanWeiss}, a countable measure-preserving equivalence relation on a standard probability space is hyperfinite if and only if it is \emph{amenable} in Zimmer's sense~\cite[p.~27]{zimmer} (see also~\cite[Definition~6]{ConnesFeldmanWeiss} or~\cite[Definition~4.57]{KerrLiBook}). The following dynamical characterization of hyperfiniteness extends a groundbreaking result of Giordano and Pestov~\cite[Theorem~5.7]{GiordanoPestov} to non-ergodic equivalence relations.

\begin{thm}[{\cite[Theorem~8.4]{lemaitre}}]\label{theorem:lemaitre} Let $E$ be a countable measure-preserving equivalence relation on a non-atomic standard probability space $(\Omega, \mu)$. The following are equivalent. \begin{enumerate}
	\item\label{theorem:lemaitre.1} $E$ is hyperfinite.
	\item\label{theorem:lemaitre.2} $[E]$ is amenable.
	\item\label{theorem:lemaitre.3} $[E]$ is extremely amenable.
	\item\label{theorem:lemaitre.4} $[E]$ is whirly amenable.
\end{enumerate} \end{thm}

\begin{proof} \ref{theorem:lemaitre.1}$\Longleftrightarrow$\ref{theorem:lemaitre.2}$\Longleftrightarrow$\ref{theorem:lemaitre.3}. This is precisely the statement of \cite[Theorem~8.4]{lemaitre}.

\ref{theorem:lemaitre.4}$\Longrightarrow$\ref{theorem:lemaitre.1}. This is trivial.

\ref{theorem:lemaitre.1}$\Longrightarrow$\ref{theorem:lemaitre.4}. This can be established via an argument as sketched in~\cite[Remark after Lemma~8.3]{lemaitre}, using~\cite[Theorem~1.1]{GlasnerTsirelsonWeiss}. We provide an alternative proof, based on~\cite[\S8]{lemaitre} and Section~\ref{section:whirliness}:

According to~\cite[Proof of Lemma~8.3]{lemaitre}, the full group of a finite measure-preserving Borel equivalence relation is isomorphic to a direct product of $L^{0}$ groups over diffuse measures and hence whirly by Theorem~\ref{theorem:L0.amenable} and Lemma~\ref{lemma:whirly.products}. As $E$ is hyperfinite, there exists an increasing sequence $(E_{n})_{n \in \N}$ of finite equivalence relations such that $E$ agrees with $\bigcup_{n \in \N} E_{n}$ almost everywhere. Necessarily, for each $n \in \N$, the equivalence relation $E_{n}$ is measure-preserving and thus $[E_{n}]$ is whirly by the preceding discussion. Since $\bigcup_{n \in \N} [E_{n}]$ is dense in $[E]$ due to~\cite[Lemma~8.2]{lemaitre}, it follows that $[E]$ is whirly thanks to Lemma~\ref{lemma:whirly.persistence}\ref{lemma:whirly.persistence.1}. \end{proof}

We conclude this section by recording two additional facts about full groups used in the proof of Theorem~\ref{theorem:examples}\ref{theorem:examples.3}. As is customary, the set of torsion elements of a group $G$ will be denoted by
\begin{displaymath}
	\Tor(G) \, \defeq \, \{ g \in G \mid \exists n \in \N_{>0} \colon \, g^{n} = 1 \}.
\end{displaymath}

\begin{lem}\label{lemma:tor.dense} Let $E$ be a countable measure-preserving equivalence relation on a standard probability space. Then $\Tor([E])$ is dense in $[E]$. \end{lem}

\begin{proof} This follows from~\cite[Proof of 494C(c)]{FremlinV4II}. \end{proof}

\begin{lem}\label{lemma:torsion} Let $E$ be a countable measure-preserving equivalence relation on a non-atomic standard probability space. If $T \in \Tor([E])$, then there exists a continuous homomorphism $\phi \colon L^{0}([0,1],\lambda;\Z) \to [E]$ such that $T = \phi(\mathbf{1})$, where $\mathbf{1} \in L^{0}([0,1],\lambda;\Z)$ denotes (the $\lambda$-equivalence class of) the constant function of value $1 \in \Z$. \end{lem}

\begin{proof} We start off by introducing an auxiliary object. If $n \in \N_{>0}$, then we let \begin{displaymath}
	I_{n} \, \defeq \, \left. \! \left\{ (t_{0},\ldots,t_{n}) \in [0,1]^{n+1} \, \right\vert 0 = t_{0} < \ldots < t_{n} = 1 \right\}
\end{displaymath} and define, for each $(t,z) \in I_{n} \times \Z^{n}$, an element $\sigma^{t,z} \in L^{0}([0,1],\lambda;\Z)$ by setting \begin{displaymath}
	\sigma_{t,z}(r) \, \defeq \, z_{k} \qquad (r \in [t_{k-1},t_{k}), \, k \in \{ 1,\ldots,n\}) .
\end{displaymath} Note that $\left(\bigcup_{n=1}^{\infty}I_{n},{\preceq}\right)$ is a directed set, where \begin{displaymath}
	t \preceq t' \ \ \ :\Longleftrightarrow \ \ \ \forall k' \in \{ 1,\ldots,n' \} \ \exists k \in \{ 1,\ldots,n \} \colon \ \, t_{k-1} \leq t'_{k'-1} \leq t'_{k'} \leq t_{k}
\end{displaymath} whenever $t \in I_{n}$, $t' \in I_{n'}$ and $n,n' \in \N_{>0}$. One readily checks that, for any $n,n' \in \N_{>0}$ and $t \in I_{n}$, $t' \in I_{n'}$ with $t \preceq t'$, there is a unique map $p_{t,t'} \colon \{ 1,\ldots,n'\} \to \{ 1,\ldots,n \}$ such that \begin{displaymath}
	\forall k' \in \{ 1,\ldots,n' \} \colon \qquad t_{p_{t,t'}(k')-1} \, \leq \, t'_{k'-1} \, \leq \, t'_{k'} \, \leq \, t_{p_{t,t'}(k')} ;
\end{displaymath} and, moreover, \begin{equation}\label{eq:refinement}
	\forall z \in \Z^{n} \colon \qquad \sigma_{t,z} \, = \, \sigma_{t',z \circ {p_{t,t'}}} .
\end{equation} Using~\eqref{eq:refinement} and \begin{equation}\label{eq:addition}
	\forall n \in \N_{>0} \ \forall z,z' \in \Z^{n} \colon \qquad \sigma_{t,z} + \sigma_{t,z'} \, = \, \sigma_{t,z + z'} ,
\end{equation} we conclude that \begin{displaymath}
	S \, \defeq \, \{ \sigma_{t,z} \mid (t,z) \in I_{n} \times \Z^{n}, \, n \in \N_{>0} \}
\end{displaymath} constitutes a subgroup of $L^{0}([0,1],\lambda;\Z)$.

Now, suppose that $E$ is given on the non-atomic standard probability space $(\Omega,\mu)$. Let $T \in \Tor([E])$. According to~\cite[Lemma 2.2]{KittrellTsankov}, there exists a map $B \colon [0,1] \to \mathcal{B}_{\mu}$ such that \begin{itemize}
	\item[---\,] $B(s) \subseteq B(t)$ for all $s,t \in [0,1]$ with $s \leq t$,
	\item[---\,] $\mu(B(t)) = t$ for every $t \in [0,1]$, and
	\item[---\,] $T(B(t)) = B(t)$ for every $t \in [0,1]$.
\end{itemize} If $n \in \N_{>0}$ and $t \in I_{n}$, then the family $\Omega^{t} \defeq (\Omega^{t}_{k})_{k \in \{ 1,\ldots,n \}}$ defined by \begin{displaymath}
	\Omega^{t}_{k} \, \defeq \, B(t_{k}) \setminus B(t_{k-1}) \qquad (k \in \{1,\ldots,n\})
\end{displaymath} constitutes a $T$-invariant partition of unity in $\mathcal{B}_{\mu}$. Consider the map $\phi \colon S \to [E]$ defined by \begin{displaymath}
	\phi(\sigma_{t,z})\vert_{\Omega^{t}_{k}} \, \defeq \, T^{z_{k}}\vert_{\Omega^{t}_{k}} \qquad (k \in \{1,\dots,n\}).
\end{displaymath} Note that $\phi$ is well defined: indeed, if $n \in \N_{>0}$ and $(t,z) \in I_{n} \times \Z^{n}$, then $T(\Omega^{t}_{k}) = \Omega^{t}_{k}$ and $T^{z_{k}} \in [E]$ for each $k \in \{ 1,\ldots,n \}$, which readily implies that $\phi(\sigma_{t,z}) \in [E]$. Evidently, $\mathbf{1} = \sigma_{(0,1),1} \in S$ and $\phi(\mathbf{1}) = T$.

We claim that $\phi$ is a continuous homomorphism. To see this, let $s,s' \in S$. Due to $\left(\bigcup_{n=1}^{\infty}I_{n},{\preceq}\right)$ being a directed set and~\eqref{eq:refinement}, there exist $n \in \N_{>0}$, $t \in I_{n}$ and $z' \in \Z^{n}$ with $s = \sigma_{t,z}$ and $s' = \sigma_{t,z'}$. For each $k \in \{ 1,\ldots,n\}$, since $T(\Omega^{t}_{k}) = \Omega^{t}_{k}$, we infer that \begin{displaymath}
	\phi(\sigma_{t,z+z'})\vert_{\Omega^{t}_{k}} \, = \, T^{z_{k}+z'_{k}}\vert_{\Omega^{t}_{k}} \, = \, \left(T^{z_{k}}\vert_{\Omega^{t}_{k}}\right)\! \left(T^{z'_{k}}\vert_{\Omega^{t}_{k}}\right) \, = \,  \phi(\sigma_{t,z})\phi(\sigma_{t,z'})\vert_{\Omega^{t}_{k}} .
\end{displaymath} It follows that \begin{displaymath}
	\phi(s+s') \, \stackrel{\eqref{eq:addition}}{=} \, \phi(\sigma_{t,z+z'}) \, = \, \phi(\sigma_{t,z})\phi(\sigma_{t,z'}) \, = \, \phi(s)\phi(s') ,
\end{displaymath} which proves that $\phi$ is a homomorphism. Moreover, note that \begin{displaymath}
	\mu(\Omega^{t}_{k}) \, = \, \mu(B(t_{k})) - \mu(B(t_{k-1})) \, = \, t_{k}-t_{k-1} \, = \, \lambda([t_{k-1},t_{k}))
\end{displaymath} for all $k \in \{ 1,\ldots,n\}$. Considering $K \defeq \{ k \in \{1,\ldots,l\} \mid z_{k} \ne z'_{k} \}$, we deduce that \begin{align*}
	d_{\mu}(\phi(s),\phi(s')) \, &\leq \, \mu\!\left( \bigcup\nolimits_{k \in K} \Omega^{t}_{k} \right) \, = \, \sum\nolimits_{k \in K} \mu(\Omega^{t}_{k}) \\
	& = \, \sum\nolimits_{k \in K} \lambda([t_{k-1},t_{k})) \, = \, \lambda\!\left( \bigcup\nolimits_{k \in K} [t_{k-1},t_{k}) \right) \, \stackrel{\ref{remark:L0.identification}}{=} \, d^0_{\lambda}(s,s') ,
\end{align*} where $d$ denotes the $\{ 0,1 \}$-valued metric on $\Z$. This shows that $\varphi$ is $1$-Lipschitz and, in particular, continuous.

Finally, since $S$ is dense in $L^{0}([0,1],\lambda;\Z)$ by~\cite[Lemma~4.1]{PestovSchneider} and $[E]$ is Ra\u{\i}kov complete by~\cite[Proposition~3.2]{KechrisBook} and Remark~\ref{remark:pestov.uspenskij}, the continuous homomorphism $\phi \colon S \to [E]$ extends to a continuous homomorphism from $L^{0}([0,1],\lambda;\Z)$ to $[E]$ (see, e.g.,~\cite[Proposition~3.6.12]{AT}), which implies the desired conclusion. \end{proof}

\section{Proofs of Theorem~\ref{theorem:examples} and Corollary~\ref{corollary:main}}\label{section:final}

\begin{proof}[Proof of Theorem~\ref{theorem:examples}] \ref{theorem:examples.1} Let $\mathcal{A}$ be a Boolean algebra and let $\mu \colon \mathcal{A} \to \R$ be a diffuse submeasure. According to~\cite[Remark~2.5]{SchneiderSolecki25}, the map $\eta \colon \Z \to S(\mu,\Z)$ defined by \begin{displaymath}
	\eta(i) \colon \, \Z \, \longrightarrow \, \mathcal{A}, \quad j \, \longmapsto \, \begin{cases}
			\, 1 & \text{if } j = i, \\
			\, 0 & \text{otherwise}
		\end{cases} \qquad (i \in \Z)
\end{displaymath} is a homomorphism. Now, let $G$ be a topological group. Due to Lemma~\ref{lemma:inert.persistence}\ref{lemma:inert.persistence.2} (resp., Lemma~\ref{lemma:whirly.persistence}\ref{lemma:whirly.persistence.2}), proving inertness (resp., whirliness) of $L^{0}(\mu,G)$ amounts to establishing the same for $S(\mu,G)$. For the latter, by Lemma~\ref{lemma:inert.persistence}\ref{lemma:inert.persistence.1} (resp., Lemma~\ref{lemma:whirly.persistence}\ref{lemma:whirly.persistence.1}), it suffices to show that every element of $S(\mu,G)$ is contained in an inert (resp., whirly) topological subgroup of $S(\mu,G)$. To this end, let $a \in S(\mu,G)$. Consider the homomorphism $\pi \colon \Z \to S(\mu,G), \, i \mapsto a^{i}$. Then, by~\cite[Lemma~2.7\,+\,Remark~2.6]{SchneiderSolecki25}, there exists a continuous homomorphism $\pi_{\#} \colon S(\mu,\Z) \to S(\mu,G)$ such that $\pi = {\pi_{\#}} \circ \eta$. In turn, \begin{displaymath}
	a \, = \, \pi(1) \, = \, \pi_{\#}(\eta(1)) \, \in \, \pi_{\#}(S(\mu,\Z)) .
\end{displaymath} Since $S(\mu,\Z)$ is inert by Theorem~\ref{theorem:sabok} and Lemma~\ref{lemma:inert.persistence}\ref{lemma:inert.persistence.2} and thus $\pi_{\#}(S(\mu,\Z))$ is inert by Lemma~\ref{lemma:inert.persistence}\ref{lemma:inert.persistence.3}, this completes the argument for inertness. Finally, if $\mu$ is a measure (or, more generally, is non-elliptic in the sense of~\cite[Definition~4.6]{SchneiderSolecki21}), then $S(\mu,\Z)$ is whirly by Theorem~\ref{theorem:L0.amenable} (or by~\cite[Corollary~7.6]{SchneiderSolecki21} in the general case) and Lemma~\ref{lemma:whirly.persistence}\ref{lemma:whirly.persistence.2} and so $\pi_{\#}(S(\mu,\Z))$ is whirly thanks to Lemma~\ref{lemma:whirly.persistence}\ref{lemma:whirly.persistence.3}, which gives the desired conclusion for whirliness.

\ref{theorem:examples.2} Let $M$ be a $\mathrm{II}_{1}$ factor. According to Lemma~\ref{lemma:whirly.persistence}\ref{lemma:whirly.persistence.1} and Lemma~\ref{lemma:inert.persistence}\ref{lemma:inert.persistence.1}, it will be sufficient to verify that every element of $\U(M)$ is contained in a whirly, inert topological subgroup of $\U(M)$. For this purpose, let $u \in \U(M)$. Then the ${}^{\ast}$-subalgebra of $M$ generated by $u$ is abelian, whence a standard application of Zorn's lemma shows that $u \in N$ for some maximal abelian ${}^{\ast}$-subalgebra $N \leq M$. Thanks to Lemma~\ref{lemma:masa}, we find a diffuse probability space $(\Omega,\mathcal{B},\mu)$ such that $L^{\infty}(\Omega,\mathcal{B},\mu;\C) \cong N$. In turn, the topological group \begin{displaymath}
	\U(L^{\infty}(\Omega,\mathcal{B},\mu;\C)) \, \stackrel{\ref{remark:unitary.group.vs.L0}}{=} \, L^{0}(\Omega,\mathcal{B},\mu;\T)
\end{displaymath} is isomorphic to $\U(N)$, and the latter is a topological subgroup of $\U(M)$ (cf.~\cite[Proposition~1.24.5, p.~78]{sakai}). Now, as $u \in \U(N)$ and \begin{displaymath}
	\U(N) \, \cong \, L^{0}(\Omega,\mathcal{B},\mu;\T) \, \stackrel{\ref{remark:L0.identification}}{\cong} \, L^{0}(\mu,\T)
\end{displaymath} is whirly and inert by Theorem~\ref{theorem:L0.amenable}, this completes the argument.

\ref{theorem:examples.3} Let $E$ be a countable measure-preserving equivalence relation on a non-atomic standard probability space. Let $\mathcal{H} \defeq \{ H \leq [E] \mid H \text{ whirly, inert} \}$. Recall that \begin{displaymath}
	L^{0}([0,1],\lambda;\Z) \, \stackrel{\ref{remark:L0.identification}}{\cong} \, L^{0}(\lambda,\Z)
\end{displaymath} is whirly and inert by Theorem~\ref{theorem:L0.amenable}. Hence, Lemma~\ref{lemma:whirly.persistence}\ref{lemma:whirly.persistence.3} and Lemma~\ref{lemma:inert.persistence}\ref{lemma:whirly.persistence.3} together assert that $\im(\phi) \in \mathcal{H}$ for every continuous homomorphism $\phi \colon L^{0}([0,1],\lambda;\Z) \to [E]$. From this and Lemma~\ref{lemma:torsion}, we infer that $\Tor([E]) \subseteq \bigcup\mathcal{H}$.
Since $\Tor([E])$ is dense in $[E]$ by Lemma~\ref{lemma:tor.dense}, this entails density of $\bigcup \mathcal{H}$ in $[E]$.
Thus, $[E]$ is whirly and inert due to Lemma~\ref{lemma:whirly.persistence}\ref{lemma:whirly.persistence.1} and Lemma~\ref{lemma:inert.persistence}\ref{lemma:inert.persistence.1}.

\emph{Alternative argument}. Once again, thanks to Lemma~\ref{lemma:whirly.persistence}\ref{lemma:whirly.persistence.1} and Lemma~\ref{lemma:inert.persistence}\ref{lemma:inert.persistence.1}, it suffices to check that every element of $[E]$ is contained in a whirly, inert topological subgroup of $[E]$. To this end, let $T \in [E]$. Choosing any representative $T_{0}$ of the $\mu$-equivalence class $T$, we see that $R \defeq \bigcup\nolimits_{n \in \Z} \gr(T_{0}^{n})$ is a countable measure-preserving equivalence relation on $(\Omega,\mu)$ and that $R$ is hyperfinite by~\cite[Theorem~1]{dye} (see also~\cite[II.6, Theorem~6.6, p.~19]{KechrisMiller}). In turn, $[R]$ is whirly and inert by Theorem~\ref{theorem:lemaitre}. Since $T \in [R]$ and $[R]$ is a topological subgroup of $[E]$, this yields the desired conclusion. \end{proof}

\begin{remark} The source of whirliness for each of the three examples established in Theorem~\ref{theorem:examples} is the phenomenon of \emph{concentration of measure}, which constitutes a central ingredient in the proofs of Theorem~\ref{theorem:L0.amenable}, \cite[Corollary~7.6]{SchneiderSolecki21}, and Theorem~\ref{theorem:lemaitre}. Due to work of Moore and Solecki~\cite[Theorem~1.1]{MooreSolecki}, the group $\CM(M,\T)$ of all continuous maps from any uncountable compact metric space $M$ to the circle $\T$, equipped with the topology of uniform convergence, possesses a whirly weakly continuous non-trivial near-action on a non-atomic standard probability space; and whirliness does not come from any form of measure concentration on the acting group in this case. Moreover, as $\CM(M,\T)$ admits an injective continuous homomorphism into the compact group $\T^{M}$, we see that the hypothesis in Corollary~\ref{corollary:whirly.vs.locally.compact} cannot be weakened to the existence of a single non-trivial whirly weakly continuous near-action. \end{remark}

\begin{proof}[Proof of Corollary~\ref{corollary:main}] Using Theorem~\ref{theorem:examples} and the material of Section~\ref{section:examples}, we find an example of a non-amenable, unitarily representable, whirly Polish group $G$, namely \begin{itemize}
	\item[---\,] $G = L^{0}(\Omega,\mathcal{B},\nu;H)$ for any non-amenable, unitarily representable Polish group $H$ and any diffuse probability space $(\Omega,\mathcal{B},\nu)$ such that the $\sigma$-algebra $\mathcal{B}$ is countably generated (cf.~Subsection~\ref{subsection:L0.groups}),
	\item[---\,] $G = \U(M)$ for any non-injective $\mathrm{II}_{1}$ factor $M$ with separable predual (cf.~Subsection~\ref{subsection:unitary.groups}), or
	\item[---\,] $G = [E]$ for any non-hyperfinite, ergodic countable measure-preserving equivalence relation $E$ on
		      a non-atomic standard probability space (cf.~Subsection~\ref{subsection:full.groups}).
\end{itemize} According to Theorem~\ref{theorem:construction}, there exist a standard probability space $(\Omega',\mu')$ and a topological group embedding $\theta \colon G \hookrightarrow \Aut(\Omega',\mu')$ such that the induced near-action $G \curvearrowright^{\theta} (\Omega',\mu')$ is ergodic. Ergodicity of $G \curvearrowright^{\theta} (\Omega',\mu')$ implies that either $(\Omega',\mu')$ is non-atomic, or $(\Omega',\mu')$ decomposes into a finite number of atoms of the same measure. However, due to $\Aut(\Omega',\mu')$ containing $\theta(G)$ and thus being infinite, the latter alternative is impossible. Hence, $(\Omega',\mu')$ is non-atomic. Thanks to~\cite[(17.41) Theorem, p.~116]{KechrisBookClassical}, there exists a Borel isomorphism $f \colon \Omega \to \Omega'$ such that $f_{\ast}(\mu) = \mu'$. In turn, \begin{displaymath}
	\iota \colon \, G \, \longrightarrow \, \Aut(\Omega,\mu), \quad g \, \longmapsto \, f \circ {\theta(g)} \circ {f^{-1}}
\end{displaymath} is a topological group embedding and the induced near-action $G \curvearrowright^{\iota} (\Omega,\mu)$ is ergodic, i.e., $\iota(G) \curvearrowright (\Omega,\mu)$ is ergodic. From Remark~\ref{remark:pestov.uspenskij} and Theorem~\ref{theorem:raikov}, we know that $\iota(G)$ is closed in $\Aut(\Omega,\mu)$. Due to the respective properties of the topological group $G$, its isomorphic copy $\iota(G)$ is non-amenable and whirly, where the latter entails whirliness of $\iota(G) \curvearrowright (\Omega,\mu)$ according to Proposition~\ref{proposition:whirly.polish}. \end{proof}

\section*{Acknowledgments}

The second-named author gratefully acknowledges inspiring conversations with Vladimir Pestov, in which he learned about Problem~\ref{problem:pestov} and Gaussian near-actions. Moreover, the authors would like to thank S\l awomir Solecki for his comments on an earlier version of the present manuscript.

\end{document}